\documentclass[3p,times]{elsarticle}
\usepackage{enumerate}

\usepackage{amsmath}

\usepackage{array}
\usepackage{amsfonts}
\usepackage{bm}
\usepackage{algorithm}
\usepackage{algorithmic}
\usepackage{amscd}
\usepackage{mathrsfs}

\usepackage{amssymb}
 \usepackage{amsthm}
 \usepackage{amsmath}
\usepackage{enumerate}
\usepackage[figuresright]{rotating}
 \usepackage[titletoc]{appendix}
 \usepackage{appendix}

\newtheorem{definition}{Definition}[section]
\newtheorem{lemma}{Lemma}[section]

\newtheorem{property}{Property}[section]
\newtheorem{theorem}{Theorem}[section]
\newtheorem{example}{Example}[section]
\newtheorem{corollary}{Corollary}[section]
\newtheorem{remark}{\em Remark}[section]

\newcommand{\PreserveBackslash}[1]{\let\temp=\\#1\let\\=\temp}
\newcolumntype{C}[1]{>{\PreserveBackslash\centering}p{#1}}
\newcolumntype{R}[1]{>{\PreserveBackslash\raggedleft}p{#1}}
\newcolumntype{L}[1]{>{\PreserveBackslash\raggedright}p{#1}}

\begin{document}

\begin{frontmatter}




\title{Variational formulation and efficient implementation for solving the tempered fractional problems
}


\author{ Weihua Deng 
 and Zhijiang Zhang }

\address{School of Mathematics and Statistics, Gansu Key Laboratory of Applied Mathematics and Complex Systems, Lanzhou University, Lanzhou 730000, P.R. China}

\begin{abstract}
Because of the finiteness of the life span and boundedness of the physical space, the more reasonable or physical choice is the tempered power-law instead of pure power-law for the CTRW model in characterizing the waiting time and jump length of the motion of particles. This paper focuses on providing the variational formulation and efficient implementation for solving the corresponding deterministic/macroscopic models, including the space tempered fractional equation and time tempered fractional equation. The convergence, numerical stability, and a series of variational equalities are theoretically proved. And the theoretical results are confirmed by numerical experiments.


\end{abstract}


\end{frontmatter}
%

\section{Introduction}

In the mesoscopic world, generally there are two types of models to describe the motion of particles, namely, the Langevin type equation and the continuous time random walk (CTRW) model, both of them being fundamental ones in statistic physics. The CTRW model is a stochastic process composed of jump lengths and waiting times with the particular probability distributions.  When the probability distribution(s) of the jump length and/or waiting time are/is power law with divergent second moment for the jump length and/or divergent first moment for the waiting times, the CTRW describes the anomalous diffusion, and its Fokker-Planck equation has space and/or time fractional derivative(s)  \cite{Metzler:00}. Nowadays, the more preferred choice for the distribution of the jump length and waiting time seems to be the tempered power-law, which makes the process very slowly converge to normal diffusion; but, most of the time, the standard normal diffusion can not be observed because of the finite life span of the biological particles. The bounded physical space urges us to use the tempered power-law distribution for the jump length. Many techniques can be used to temper the power-law distribution, such as, discarding the very large jumps directly \cite{Mantegna:94}, adding a high order power-law factor \cite{Sokolov:04} or a nonlinear friction term \cite{Chechkin:05}. Exponentially tempering the power-law distributions seems to be the most popular one \cite{Cartea:07,Meerschaert:09}, which has both the mathematical and technique advantages \cite{Baeumera:10,Sabzikar:15}; and the probability densities of the tempered stable process solve the tempered fractional equation.


For extending and digging out the potential applications of the tempered dynamics, it is necessary to efficiently solve the corresponding deterministic/macoscopic tempered equation, which is the issue this paper is focusing on. In fact, there are already a lot of research works for numerically solving the (non-tempered) fractional partial differential equations (PDEs); almost all of the numerical methods for classical PDEs are extended to the fractional ones, including the finite difference method \cite{Meerschaert:04,Tian:15,Zhaox:15}, the finite element  or  discontinuous finite element method \cite{Jin:15,Deng:08,Deng:13,Ervin:05,Ervin:07,Mustapha:14,Wang:13}, the spectral or spectral element method \cite{Li:09,Li:10,Zayernouri:13}; and the connection of fractional PDEs with nonlocal problem is discussed in \cite{Defterli:15}.
Mathematically, fractional calculus \cite{Podlubny:99} is the special case of the tempered fractional calculus with the parameter $\lambda=0$. And the definition of the tempered fractional calculus is much similar to the one of the fractional substantial calculus \cite{Carmi:10}, but they come from the completely different physical background. The research works of numerical methods for tempered fractional PDEs are very limited. In \cite{Baeumera:10,Chen:15,Lican:15, Sabzikar:15}, the finite difference methods are proposed to solve the tempered space fractional PDEs.  Hanert and Piret in \cite{Hanert:14} consider the Chebyshev pseudospectral method for the space-time tempered fractional diffusion equations. More recently, Zayernouri, Ainsworth, and Karniadakis \cite{Zayernourt:15} investigate the tempered fractional Sturm-Liouville eigenproblems.
The efforts made by this paper can be summarized as two aspects. The first one is to develop the variational space that works for the tempered fractional operators, which can be regarded as the generalization of the theory presented in \cite{Ervin:05,Li:09} for the fractional differential operators; based on the space, the Galerkin and Petro-Galerkin finite element methods get their theoretical framework for solving the tempered fractional PDEs; and the variational properties of the tempered fractional operators are discussed, which should also be useful for the theoretical analysis of discontinuous Galerkin method \cite{Deng:13,Qiu:15,Xu:14} for the PDEs involving the tempered fractional calculus. The second one is focusing on the application of the developed theory and the efficient implementation of the proposed schemes; the implementation details are carefully discussed, and the efficiency is analyzed and illustrated.

The rest of this paper is organized as follows. In Section 2, we introduce some basic  definitions and properties of the tempered fractional calculus, and derive some  essential inequalities. In Section 3,  we provide the variational formulation and derive the variational equalities and inequalities involving the tempered fractional operators.
 Then, in Section 4, we apply the developed framework to solve the space tempered and time tempered fractional PDEs, in particular, the convergence and stability analysis, and the efficient numerical implementation are detailedly discussed.
 The numerical results, presented in Section 5, confirm the computational efficiency of the proposed numerical schemes. Finally, we conclude the paper with some remarks.


\section {Preliminaries: definitions and lemmas to be used}
We start with some definitions and properties of the tempered fractional integrals and derivatives \cite{Baeumera:10,Cartea:07,Sabzikar:15}. In this paper, we use ${}_aD_x^{-\mu}u(x)$ and ${}_xD_b^{-\mu}u(x)$, ${}_a D_x^{\mu}u(x)$ and ${}_xD_b^{\mu}u(x)$, and ${}_a^C D_x^{\mu}u(x)$  to denote the standard left and right Riemann-Liouville fractional integrals,  the  standard left and right Riemann-Liouville fractional derivatives, and the left Caputo fractional derivative of order $\mu$ on $(a,b)$,  respectively, which can be found in \cite{Podlubny:99}. Of course, $(a,b)$ can also be $\mathbb{R}=(-\infty, \infty)$.

\begin{definition}\label{definition:1}
For any $\mu\ge0$  and  fixed parameter $\lambda\ge 0$, the left and right tempered  Riemann-Liouville fractional integrals of function $u(x)$ on $(a,b)$  are, respectively,  defined by
\begin{eqnarray}
&&{}_a\mathbb{D}_x^{-\mu,\lambda}u(x):=e^{-\lambda x}{}_aD_x^{-\mu}\left(e^{\lambda x}u(x)\right)=\int_a^x\frac{(x-\xi)^{\mu-1}}{\Gamma(\mu)}e^{-\lambda(x-\xi)}u(\xi)d\xi,
\end{eqnarray}
and
\begin{eqnarray}
&&{}_x\mathbb{D}_b^{-\mu,\lambda}u(x):=e^{\lambda x}{}_xD_b^{-\mu}\left(e^{-\lambda x}u(x)\right)=\int_x^b\frac{(\xi-x)^{\mu-1}}{\Gamma(\mu)}e^{-\lambda(\xi-x)}u(\xi)d\xi.
\end{eqnarray}
\end{definition}

\begin{definition} \label{definition:2}
For any $n-1\le\mu<n \,(n\in \mathbb{N}^+)$ and fixed parameter $\lambda \ge 0$,  define
\begin{eqnarray}\label{defein2:1}
&&{}_a \mathbb{D}_x^{\mu,\lambda}u(x):=e^{-\lambda x}{}_aD_x^{\mu}\left(e^{\lambda x}u(x)\right)=\frac{e^{-\lambda x}}{\Gamma(n-\mu)}\frac{d^n}{dx^n}\int_a^x\frac{e^{\lambda\xi}u(\xi)}{(x-\xi)^{\mu-n+1}}d\xi,
\end{eqnarray}
and
\begin{eqnarray}\label{defein2:2}
&&{}_x\mathbb{D}_b^{\mu,\lambda}u(x):=e^{\lambda x}{}_xD_b^{\mu}\left(e^{-\lambda x}u(x)\right)=\frac{e^{\lambda x}}{\Gamma(n-\mu)}(-1)^n\frac{d^n}{dx^n}\int_x^b\frac{e^{-\lambda\xi}u(\xi)}{(\xi-x)^{\mu-n+1}}d\xi.
\end{eqnarray}
Then for $1<\mu\le 2$, the left and right tempered Riemann-Liouville  fractional derivatives of function $u(x)$    on $(a,b)$ are, respectively, defined by
\begin{eqnarray}
&&{}_a D_x^{\mu,\lambda}u(x):={}_a \mathbb{D}_x^{\mu,\lambda}u(x)-\lambda^{\mu}u(x)-\mu \lambda^{\mu-1}\frac{du(x)}{d x},
\end{eqnarray}
and
\begin{eqnarray}
&&{}_xD_b^{\mu,\lambda}u(x):={}_x\mathbb{D}_b^{\mu,\lambda}u(x)-\lambda^{\mu}u(x)+\mu \lambda^{\mu-1}\frac{du(x)}{dx}.
\end{eqnarray}
\end{definition}

The tempered   fractional derivative can also be given in the Caputo sense.
\begin{definition}\label{definition:3}
For any $n-1\le\mu<n \, (n\in \mathbb{N}^+)$ and fixed parameter $\lambda\ge 0$, the left tempered Caputo  fractional derivative of function $u(x)$    on $(a,b)$  is defined by
\begin{eqnarray}
{}_a^C\mathbb{D}_t^{\mu,\lambda}u(x):=e^{-\lambda x}\,{}^C_a\mathbb{D}^{\mu}_x\left(e^{\lambda x}u(x)\right)=\frac{e^{-\lambda x}}{{\Gamma(n-\mu)}}\int_a^x\frac{e^{\lambda\xi}}{(x-\xi)^{\mu-n+1}}\left(\frac{d}{dx}+\lambda\right)^n u(\xi)d\xi.
\end{eqnarray}
\end{definition}

If $\lambda=0$, the tempered fractional integrals and derivatives in Definitions \ref{definition:1}, \ref{definition:2} and \ref{definition:3} all reduce to the corresponding standard Riemann-Liouville or Caputo fractional integrals and derivatives \cite{Podlubny:99}. Noting that
\begin{eqnarray} \label{definitionequ}
~~~~~\frac{d(e^{\lambda x}f(x))}{dx}=e^{\lambda x}\left(\frac{d}{dx}+\lambda\right)f(x) {~~\rm and ~}\frac{d(e^{-\lambda x}f(x))}{dx}=e^{-\lambda x}\left(\frac{d}{dx}-\lambda\right)f(x),
\end{eqnarray}
for $n\in \mathbb{N}^+$, it is easy to check that
\begin{eqnarray}
&&{}_a \mathbb{D}_x^{n,\lambda}u(x)=e^{-\lambda x}\frac{d^n\left(e^{\lambda x}u\right)}{d x^n}=\left(\frac{d}{dx}+\lambda\right)^nu(x),\\
&&{}_x \mathbb{D}_b^{n,\lambda}u(x)=(-1)^ne^{\lambda x}\frac{d^n\left(e^{-\lambda x}u\right)}{d x^n}=(-1)^n\left(\frac{d}{dx}-\lambda\right)^nu(x).
\end{eqnarray}
Moreover, it holds that
\begin{eqnarray}
&&{}_a \mathbb{D}_x^{\mu,\lambda}u(x)=\left(\frac{d}{dx}+\lambda\right)^n{}_a\mathbb{D}_x^{-(n-\mu),\lambda}u(x),\label{equationee2}\\
&&{}_x \mathbb{D}_b^{\mu,\lambda}u(x)=(-1)^n\left(\frac{d}{dx}-\lambda\right)^n{}_x\mathbb{D}_b^{-(n-\mu),\lambda}u(x),\label{equationee21}
\end{eqnarray}
 which can be obtained by continuously apply  (\ref {definitionequ}) to  the right-sides of (\ref{equationee2}) and (\ref{equationee21}). Let $\lambda=pU(y)$. Then, they actually become the fractional substantial derivatives defined in \cite{Chen:115, Carmi:10,Deng:15}.

If $u(x)$ possesses $(n-1)$-th derivative at $ a$, one has
\begin{eqnarray}\label{equationee1}
{}_a^C\mathbb{D}_{x}^{\mu,\lambda}u(x)&=&{}_a^C\mathbb{D}_x^{\mu,\lambda}\Big[u(x)-T_{n-1}[u;a]\Big]={}_a\mathbb{D}_x^{\mu,\lambda}\Big[u(x)-T_{n-1}[u;a]\Big]\\
&=&{}_a\mathbb{D}_x^{\mu,\lambda}u(x)-\sum\limits_{k=0}^{n-1}\frac{e^{-\lambda (x-a)}(x-a)^{k-\mu}}{\Gamma(k-\mu+1)}\,{\rm D}_x^ku(x)\Big|_{x=a},\nonumber
\end{eqnarray}
where $ T_{n-1}[u;a]=\sum\limits_{k=0}^{n-1}\frac{e^{-\lambda (x-a)}(x-a)^{k}}{\Gamma(k+1)}{\rm D}_x^ku(x)\Big|_{x=a}$ and ${\rm D}_x^ku(x)=\left(\frac{d}{dx}+\lambda\right)^k u(x)$. Therefore, ${}_a\mathbb{D}_x^{\mu,\lambda}$ and ${}_a^C\mathbb{D}_{x}^{\mu,\lambda}u(x)$  coincide with each other while  ${\rm D}_x^ku(a)=0, k=0,\cdots,n-1$.

The adjoint property of the standard Riemann-Liouville integrals \cite{Deng:08,Podlubny:99}
still holds for their tempered counterparts, i.e.,
\begin{eqnarray}\label{adjonteq}
\left({}_a\mathbb{D}_x^{-\mu,\lambda}u,v\right)=\left({}_a{D}_x^{-\mu}\left(e^{\lambda x}u(x)\right),e^{-\lambda x}v(x)\right)=\left(u(x),{}_x\mathbb{D}_b^{-\mu,\lambda}v(x)\right),
\end{eqnarray}
where $(\cdot, \cdot)$ denotes the   inner product in $L^2$ sense. And by the composition rules of the standard Riemann-Liouville integrals \cite[p. 67-68] {Podlubny:99}, one also has
\begin{eqnarray}\label{semigroup}
&&{}_a\mathbb{D}_x^{-\mu_1,\lambda}{}_a\mathbb{D}_x^{-\mu_2,\lambda}u(x)={}_a\mathbb{D}_x^{-(\mu_1+\mu_2),\lambda}u(x)~~ \forall \mu_1,\mu_2 >0,\\
&&{}_x\mathbb{D}_b^{-\mu_1,\lambda}{}_x\mathbb{D}_b^{-\mu_2,\lambda}u(x)={}_x\mathbb{D}_b^{-(\mu_1+\mu_2),\lambda}u(x)~~\forall \mu_1,\mu_2 >0.\nonumber
\end{eqnarray}

\begin{property} \label{propertyeq1}
Let $u\in L^2(\Omega), n-1<\mu<n\,\, (n\in \mathbb{N}^+)$. Then
\begin{eqnarray} \label{property22eq1}
{}_a\mathbb{D}_x^{\mu,\lambda}\,{}_a\mathbb{D}_x^{-\mu,\lambda}u(x)=u(x), ~~~{}_x\mathbb{D}_b^{\mu,\lambda}\,{}_x\mathbb{D}_b^{-\mu,\lambda}u(x)=u(x).
\end{eqnarray}
Further, suppose that $u(x)$ is $n-1$ times continuously differentiable  and its $n$-th derivative is integrable, and
$\frac{d^k u(x)}{dx^k} \big|_{x=a}=0$\, $\left(\frac{d^k u(x)}{dx^k} \big|_{x=b}=0\right)$ for $k=0,\cdots,n-1$. Then
\begin{eqnarray}\label{propertyeq3}
{}_a\mathbb{D}_x^{-\mu,\lambda}\,{}_a\mathbb{D}_x^{\mu,\lambda}u(x)=u(x) ~~~\left( {}_x\mathbb{D}_b^{-\mu,\lambda}\,{}_x\mathbb{D}_b^{\mu,\lambda}u(x)=u(x) \right).
\end{eqnarray}
\end{property}
\begin{proof} Here we just prove the results for the left tempered fractional operator. The ones for the right tempered fractional operator can be similarly got.
 By
 $$
\begin{array}{lll}
{}_a\mathbb{D}_x^{\mu,\lambda}\,{}_a\mathbb{D}_x^{-\mu,\lambda}u(x)&=&e^{-\lambda x}{}_a{D}_x^{\mu}\left[e^{\lambda x}\left({}_a\mathbb{D}_x^{-\mu,\lambda}u(x)\right)\right]\\
&=&e^{-\lambda x}{}_a{D}_x^{\mu}\left[{}_a{D}_x^{-\mu}\left(e^{\lambda x} u(x)\right)\right]=u(x),\nonumber
\end{array}
$$
one ends the proof of (\ref{property22eq1}). Further, noting that
$$
\begin{array}{lll}
{}_a\mathbb{D}_x^{-\mu,\lambda}\,{}_a\mathbb{D}_x^{\mu,\lambda}u(x) &=& e^{-\lambda x}{}_a{D}_x^{-\mu}\left[e^{\lambda x}\left({}_a\mathbb{D}_x^{\mu,\lambda}u(x)\right)\right]\\
&=& e^{-\lambda x}{}_a{D}_x^{-\mu}\left[{}_a{D}_x^{\mu}\left(e^{\lambda x} u(x)\right)\right],\nonumber
\end{array}
$$
from the discussion of \cite[p. 75-77] {Podlubny:99}, we know that (\ref{propertyeq3}) holds if $\frac{d^k(e^{\lambda x} u(x))}{dx^k}\Big|_{x=a}=0$ for $k=0,\cdots,n-1$, which follows directly after using $\frac{d^k(e^{\lambda x} u(x))}{dx^k}=e^{\lambda x}\left(\frac{d}{dx}+\lambda\right)^ku(x)=e^{\lambda x} {\rm D}_x^k u(x)$ and $\frac{d^k u(x)}{dx^k} \big|_{x=a}=0$ for $k=0,\cdots,n-1$.
\end{proof}

\begin{property} [see \cite{Baeumera:10,Chen:115}]\label{propertyeq2}
For $u\in L^2(\mathbb{R})$ and $\mu\ge 0$, it holds that
\begin{eqnarray*}
&&\mathscr{F}[{}_{-\infty}\mathbb{D}_x^{-\mu,\lambda}u(x)](\omega)=(\lambda+i\omega)^{-\mu}\mathscr{F}[u](\omega),\\
&&\mathscr{F}[{}_x\mathbb{D}_{\infty}^{-\mu,\lambda}u(x)](\omega)=(\lambda-i\omega)^{-\mu}\mathscr{F}[u](\omega).
\end{eqnarray*}
If $u\in C_0^{\infty}(\mathbb{R})$ further, then
\begin{eqnarray*}
&&\mathscr{F}[{}_{-\infty}\mathbb{D}_x^{\mu,\lambda}u(x)](\omega)=(\lambda+i\omega)^\mu\mathscr{F}[u](\omega),\\
&&\mathscr{F}[{}_\mu\mathbb{D}_{\infty}^{\mu,\lambda}u(x)](\omega)=(\lambda-i\omega)^\mu\mathscr{F}[u](\omega).
\end{eqnarray*}
Here $\mathscr{F}[u](\omega)=\int_{-\infty}^{\infty}e^{-i\omega x}u(x)dx$ denotes the Fourier transform of $u(x)$.
\end{property}

\begin{lemma}\label{lemmaeq1}
Let $x\ge0$. Then
\begin{equation}
{2^{\mu-1}}{\left(1+x^{\mu}\right)}\le (1+x)^{\mu}\le \left(1+x^{\mu}\right)\quad 0<\mu\le1
\end{equation}
and
\begin{equation}
\left(1+x^{\mu}\right)\le (1+x)^{\mu}\le 2^{\mu-1}\left(1+x^{\mu}\right)\quad \mu>1.
\end{equation}
\end{lemma}
\begin{proof}
Noting that  $g(x)=x^{\mu}$ is concave for $\mu \in (0,1]$ and convex for $\mu>1$ , one has
 \begin{eqnarray}
 \left\{\begin{array}{ll}\frac{1+x^{\mu}}{2}\le \left(\frac{1+x}{2}\right)^{\mu}& 0<\mu\le 1,\\
 \frac{1+x^{\mu}}{2}\ge \left(\frac{1+x}{2}\right)^{\mu}&\mu>1.
 \end{array}\right.
 \end{eqnarray}
Then using the fact that  $g(x)=1+x^{\mu}-(1+x)^{\mu}, x\in [0,\infty)$ is increasing for $\mu\in (0,1]$ and decreasing for $\mu >1$, the proof is completed.
\end{proof}

\begin{lemma} \label{lemmaeq2}
Let $\hat{\theta} \in [0,\frac{\pi}{2}]$ and $1<\alpha\le2$. Then
\begin{eqnarray}
\sin(\alpha\hat{\theta})\ge\sin(\hat{\theta})\cos^{\alpha-1}(\hat{\theta}),
\end{eqnarray}
 where  ``$=$" holds if and only if $\hat{\theta}=0$ if $1<\alpha<2$, and $\hat{\theta}=0$ or $\frac{\pi}{2}$ if $\alpha=2$.
\end{lemma}
\begin{proof} That the inequality holds can be easily checked for the case $\alpha=2$. Now, we prove the case $1<\alpha<2$. The inequality obviously holds if $\hat{\theta}=0$ or $\frac{\pi}{2}$. In the following, we assume that $\hat{\theta}\in (0,\frac{\pi}{2})$. Note that
\begin{eqnarray*}
&&\sin(\alpha \hat{\theta})-\sin(\hat{\theta})\cos^{\alpha-1}(\hat{\theta})\\
&&=\sin\left((\alpha-1)\hat{\theta}+\hat{\theta}\right)-\sin(\hat{\theta})\cos^{\alpha-1}(\hat{\theta})\\
&& > \sin(\hat{\theta})\left(\cos\left((\alpha-1)\hat{\theta}\right)-\cos^{\alpha-1}(\hat{\theta})\right).
\end{eqnarray*}

Letting $g(\hat{\theta})=\cos(\beta \hat{\theta})-\cos^{\beta}(\hat{\theta})$ with $\beta\in (0,1)$ and $\hat{\theta}\in (0,\frac{\pi}{2})$, then
$$
\begin{array}{lll}
g^{\prime}(\hat{\theta}) &=& -\beta \sin(\beta\hat{\theta})+\beta\cos^{\beta-1}(\hat{\theta})\sin(\hat{\theta})\\
 &=& -\beta \sin(\hat{\theta})\left(\frac{\sin(\beta\hat{\theta})}{\sin(\hat{\theta})}-\cos^{\beta-1}(\hat{\theta})\right).
\end{array}
$$
And for $\beta\in(0,1)$ and $\hat{\theta}\in (0,\frac{\pi}{2})$,  there exists $ \frac{\sin(\beta \hat{\theta})}{\sin(\hat{\theta})}< 1, \cos^{\beta-1}(\hat{\theta})>1$. Therefore, $g(\hat{\theta})$ is strictly increasing in $[0,\frac{\pi}{2})$. Then we arrive at the conclusion.
\end{proof}

In the rest of this paper,  we will use $\Omega=(a,b)$ to denote a finite interval. By $A\stackrel{<}{\sim}B$, we mean that $A$ can be bounded by a multiple of $B$, independent of the parameters they may depend on. And the expression $A\sim B$ means that $A\stackrel{<}{\sim}B\stackrel{<}{\sim}A$.

\section{Variational formulation and its related properties for the tempered fractional calculus
}
To develop the variational method for solving the tempered fractional PDEs, one needs to develop the variational formation and discuss its related properties for the tempered fractional calculus, being the issues this section is dealing with.

For any $\mu\ge0$, let  $H^{\mu}(\mathbb{R})$  be the Sobolev space of  order $\mu$ on $\mathbb{R}$,  and ${H}^{\mu}(\Omega)$ denotes the space of restrictions of the functions from $H^{\mu}(\mathbb{R})$.  More specifically,
\begin{equation}
H^\mu(\mathbb{R})=\left\{u(x)\in L^2(R)\,\big|\left|u\right|^2_{H^\mu(\mathbb{R})}<\infty\right\}
\end{equation}
endowed with the seminorm
\begin{equation}
\left|u\right|^2_{H^\mu(\mathbb{R})}= \int_{\mathbb{R}}|\omega|^{2\mu}\left|\mathscr{F}[u](\omega)\right|^2d\omega
\end{equation}
and the norm
\begin{equation}
\left\|u\right\|^2_{H^\mu(\mathbb{R})}=
\int_{\mathbb{R}}\left(1+|\omega|^{2\mu}\right)\left|\mathscr{F}[u](\omega)\right|^2d\omega\sim \int_{\mathbb{R}}\left(1+|\omega|^2\right)^{\mu}\left|\mathscr{F}[u](\omega)\right|^2d\omega;
\end{equation}
\begin{equation}
H^{\mu}(\Omega)=\left\{u\in L^2(\Omega)\,\big| \exists \tilde{u}\in H^\mu(\mathbb{R}) {~\rm such ~that~} \tilde{u}|_{\Omega}=u\right\}
\end{equation}
endowed with
\begin{eqnarray}
|u|_{H^\mu(\Omega)}=\inf_{\tilde{u}|_{\Omega}=u}\left|\tilde{u}\right|_{H^\mu(\mathbb{R})}
{~\rm and~}
 \left\|u\right\|^2_{H^\mu(\Omega)}=\left\|u\right\|^2_{L^2(\Omega)}+|u|^2_{H^\mu(\Omega)}.
 \end{eqnarray}
There are also some other definitions of the fractional Sobolev space; for the equivalence between them refer to \cite{Adams:75,Tartar:07}. $H^{\mu}_0(\Omega)$ denotes the closure of $C_0^{\infty}(\Omega)$ w.r.t. $\|\cdot\|_{H^\mu(\Omega)}$.
We first list the following fractional Poincar{\'e}-Friedrichs inequality and the embeddedness, which can be found in  \cite[Corollary 2.15]{Ervin:05}.
\begin{lemma}\label{lemmaeq3}
Let $0<\mu_1<\mu_2$, and $\mu_1, \mu_2\not=n-\frac{1}{2}\, (n\in \mathbb{N^+})$. If $u\in H_0^{\mu_2}(\Omega)$, one has
\begin{eqnarray}
\left\|u\right\|_{L^2(\Omega)}\stackrel{<}{\sim}\left|u\right|_{H^{\mu_2}_0(\Omega)}{~\rm and~} \left|u\right|_{H^{\mu_1}(\Omega)}\stackrel{<}{\sim}\left|u\right|_{H^{\mu_2}(\Omega)}.
\end{eqnarray}
\end{lemma}
In the following, we will focus on the case $\mu\in (0,1]$ but $\mu \not=\frac{1}{2}$.
\begin{theorem}\label{theorem11}
For any $0<\mu\le1$ and fixed parameter $\lambda\ge0$, the operators ${}_a \mathbb{D}_x^{\mu,\lambda}u(x)$ and $ {}_x\mathbb{D}_b^{\mu,\lambda}u(x)$ defined for $u\in C_0^{\infty}(\Omega)$ can be continuously extended to operators from $H_0^{\mu}(\Omega)$ to $L^2(\Omega)$.
\end{theorem}
\begin{proof}
 First, for $u(x)\in C_0^{\infty}(\mathbb{R})$, by Property \ref {propertyeq2} and Plancherel's theorem, one has
\begin{eqnarray}
&&\left\|{}_{-\infty}\mathbb{D}_x^{\mu,\lambda}u\right\|^2_{L^2(\mathbb{R})}=\left\|{}_{x}\mathbb{D}_\infty^{\mu,\lambda}u\right\|^2_{L^2(\mathbb{R})}
=\frac{1}{2\pi}\int_{\mathbb{R}}\left(\lambda^2+|\omega|^2\right)^{\mu}\left|\mathscr{F}[u](\omega)\right|^2d\omega.
\end{eqnarray}
If $\lambda\not=0$, by Lemma \ref{lemmaeq1} one has
\begin{eqnarray}\label{theorem1eq4}
\left(\lambda^2+|\omega|^2\right)^{\mu}=\lambda^{2\mu}\left(1+\frac{|\omega|^2}{\lambda^2}\right)^{\mu}\sim \left(\lambda^{2\mu}+|\omega|^{2\mu}\right).
\end{eqnarray}
Note that
\begin{equation}\label{theroem1eq3}
\min\left\{1,\lambda^{2\mu}\right\}\left(1+|\omega|^{2\mu}\right)\le \left(\lambda^{2\mu}+|\omega|^{2\mu}\right)\le \max\left\{1,\lambda^{2\mu}\right\}\left(1+|\omega|^{2\mu}\right).
\end{equation}
Therefore,
\begin{eqnarray}\label{theoremeq1}
\left\|{}_{-\infty}\mathbb{D}_x^{\mu,\lambda}u\right\|^2_{L^2(\mathbb{R})}=\left\|{}_{x}\mathbb{D}_\infty^{\mu,\lambda}u\right\|^2_{L^2(\mathbb{R})}
\stackrel{<}{\sim}\left\|u\right\|^2_{H^\mu(\mathbb{R})}.
\end{eqnarray}
For $\lambda=0$, (\ref{theoremeq1}) holds obviously. 

Secondly, for $u\in C_0^{\infty}(\Omega)$, let $\tilde{u}$ defined on $\mathbb{R}$ be the zero extension of $u$. From (\ref{theoremeq1}), one has
 \begin{eqnarray}
\left\|{}_{-\infty}\mathbb{D}_x^{\mu,\lambda}\tilde{u}\right\|^2_{L^2(\mathbb{R})}=\left\|{}_{x}\mathbb{D}_\infty^{\mu,\lambda}\tilde{u}\right\|^2_{L^2(\mathbb{R})}
\stackrel{<}{\sim}\left\|\tilde{u}\right\|^2_{H^\mu(\mathbb{R})}\sim \left\|{u}\right\|^2_{H^\mu(\Omega)}.
\end{eqnarray}
Noting that  ${}_{-\infty}\mathbb{D}_x^{\mu,\lambda}\tilde{u}\big|_{\Omega}={}_{a}\mathbb{D}_x^{\mu,\lambda}u$ and ${}_{x}\mathbb{D}_\infty^{\mu,\lambda}\tilde{u}\big|_{\Omega}={}_{x}\mathbb{D}_b^{\mu,\lambda}{u}$, it yields that
\begin{eqnarray}\label{normeq1}
\left\|{}_{a}\mathbb{D}_x^{\mu,\lambda}{u}\right\|^2_{L^2(\Omega)}\stackrel{<}{\sim}\left\|{u}\right\|^2_{H^\mu(\Omega)}{~\rm and~} \left\|{}_{x}\mathbb{D}_b^{\mu,\lambda}{u}\right\|^2_{L^2(\Omega)}\stackrel{<}{\sim}\left\|{u}\right\|^2_{H^\mu(\Omega)}.
\end{eqnarray}
Then the conclusion follows after using the density of $C_0^{\infty}(\Omega)$ in $H_0^{\mu}(\Omega)$.
\end{proof}

Now, ${}_a \mathbb{D}_x^{\mu,\lambda}$ and $ {}_x\mathbb{D}_b^{\mu,\lambda}$ make sense in  $H^{\mu}_0(\Omega)$, which map $H_0^{\mu}(\Omega)$ to $L^2(\Omega)$; and the norms satisfy (\ref{normeq1}). In fact, for any $u\in H^{\mu}_0(\mathbb{R})$ and $\lambda \not= 0$, by (\ref{theorem1eq4}) and (\ref{theroem1eq3}), it holds that \begin{eqnarray}
\left\|{}_{-\infty}\mathbb{D}_x^{\mu,\lambda}u\right\|^2_{L^2(\mathbb{R})}\sim\left\|{}_{x}\mathbb{D}_\infty^{\mu,\lambda}u\right\|^2_{L^2(\mathbb{R})}
\sim \left\|u\right\|^2_{H^\mu(\mathbb{R})}.
\end{eqnarray}
In the following, we  give the similar results  for $u\in H_0^{\mu}(\Omega)$.

\begin{theorem}\label{theorem2}
For real functions $u(x)$ and $v(x)$ belonging to $H_0^{\mu}(\Omega)$ and $\lambda\ge 0$, define
\begin{eqnarray}
B(u,v):=\left({}_{a}\mathbb{D}_x^{\mu,\lambda}{u},{}_{x}\mathbb{D}_b^{\mu,\lambda}{v}\right) \quad {\rm for}~~0<\mu<\frac{1}{2}
 \end{eqnarray}
 and
 \begin{eqnarray}\label{eqnarreq2ee}
 B(u,v):=-\left({}_{a}\mathbb{D}_x^{\mu,\lambda}{u},{}_{x}\mathbb{D}_b^{\mu,\lambda}{v}\right)+(1+c_0)\lambda^{2\mu}\left(u,v\right) \quad  {\rm for}~~ \frac{1}{2}<\mu\le1,
\end{eqnarray}
where $c_0$ is any given positive constant. Then
\begin{eqnarray}\label{eeeddd}
B(u,u)\sim \left\|u\right\|^2_{H^\mu(\Omega)}\sim \left\|{}_{a}\mathbb{D}_x^{\mu,\lambda}{u}\right\|^2_{L^2(\Omega)}\sim \left\|{}_{x}\mathbb{D}_b^{\mu,\lambda}{u}\right\|^2_{L^2(\Omega)}.
\end{eqnarray}
\end{theorem}
\begin{proof}
It is enough to prove the case $u\in C_0^{\infty}(\Omega)$. Denote the zero extension of $u$ by $\tilde{u}$ defined on $\mathbb{R}$.
Then
\begin{eqnarray}\label{thereorm2eq1}
\left({}_{a}\mathbb{D}_x^{\mu,\lambda}{u},{}_{x}\mathbb{D}_b^{\mu,\lambda}{u}\right)
=\left({}_{-\infty}\mathbb{D}_x^{\mu,\lambda}{\tilde{u}},{}_{x}\mathbb{D}_\infty^{\mu,\lambda}{\tilde{u}}\right)_{L^2(\mathbb{R})}
{~\rm and~}\left(u,u\right)= \left(\tilde{u},\tilde{u}\right)_{L^2(\mathbb{R})}.
\end{eqnarray}
Note that
\begin{equation}
\overline{(\lambda+i\omega)^{\mu}}=
\left\{\begin{array}{lc}
\exp(-i2\mu \theta)\overline{(\lambda-i\omega)^{\mu}} &~\mbox{if}~ \omega\ge0,\\
\exp(i2\mu \theta)\overline{(\lambda-i\omega)^{\mu}} &~\mbox {if}~\omega<0,
\end{array} \right.
\end{equation}
where $\overline{(\cdot)}$ denotes complex conjugate, $i=\sqrt{-1}$, and
\begin{equation}\label{thetaeq}
\theta=
\left\{\begin{array}{lc}
\arctan(\frac{|w|}{\lambda})\in [0,\frac{\pi}{2}] &\mbox{if}~\lambda>0,\\
\frac{\pi}{2}, &\mbox {if} ~\lambda=0.
\end{array} \right.
\end{equation}
By Property \ref{propertyeq2} and the Plancherel theorem, it holds that
\begin{eqnarray}\label{thereorm2eq2}
&&2\pi\cdot\left({}_{-\infty}\mathbb{D}_x^{\mu,\lambda}{\tilde{u}},{}_{x}\mathbb{D}_\infty^{\mu,\lambda}{\tilde{u}}\right)_{L^2(\mathbb{R})}=
\int_{\mathbb{R}}(\lambda+i\omega)^{\mu}\mathscr{F}[{\tilde{u}}](\omega)\overline{(\lambda-i\omega)^{\mu}
\mathscr{F}[{\tilde{u}}](\omega)}d\omega\\
&&=\int_{-\infty}^0\Big|(\lambda+i\omega)^{\mu}\mathscr{F}[{\tilde{u}}](\omega)\Big|^2\exp(-i2\mu \theta)d\omega
+\int_{0}^{\infty}\Big|(\lambda+i\omega)^{\mu}\mathscr{F}[{\tilde{u}}](\omega)\Big|^2\exp(i2\mu \theta)d\omega\nonumber\\
&&=\int_{-\infty}^\infty\Big|(\lambda+i\omega)^{\mu}\mathscr{F}[{\tilde{u}}](\omega)\Big|^2\cos(2\mu \theta)d\omega\nonumber\\
&&~+i\left(\int_{0}^{\infty}\sin(2\mu \theta)\Big|(\lambda+i\omega)^{\mu}\mathscr{F}[{\tilde{u}}](\omega)\Big|^2d\omega-
\int_{-\infty}^0\sin(2\mu \theta)\Big|(\lambda+i\omega)^{\mu}\mathscr{F}[{\tilde{u}}](\omega)\Big|^2d\omega\right)\nonumber\\
&&=\int_{-\infty}^\infty\Big|(\lambda+i\omega)^{\mu}\mathscr{F}[{\tilde{u}}](\omega)\Big|^2\cos(2\mu \theta)d\omega,\nonumber
\end{eqnarray}
where in the last step $\overline{\mathscr{F}(\tilde{u})(-\omega)}=\mathscr{F}(\tilde{u})(\omega)$ has been used.

For  $\lambda=0$, one always has $B(u,u)=\frac{\left|\cos(\pi\mu)\right|}{2\pi}\left\|{}_{-\infty}\mathbb{D}_x^{\mu,\lambda}\tilde{u}\right\|^2_{L^2(\mathbb{R})}$;   combining with Lemma \ref{lemmaeq3}, then (\ref{eeeddd}) is obtained.
In the following, we assume that $\lambda \not=0$; and then $\theta$ depends on $\omega$.

For $\mu\in (0,\frac{1}{2})$, one has $2\mu\theta\in[0,\frac{\pi}{2})$ and $0<\cos(\pi\mu)\le\cos(2\mu\theta)\le1$. Therefore,
\begin{eqnarray}\label{eqeeeee}
B( {u},{u})\sim\left\|{}_{-\infty}\mathbb{D}_x^{\mu,\lambda}\tilde{u}\right\|^2_{L^2(\mathbb{R})}\sim \left\|\tilde{u}\right\|^2_{H^\mu(\mathbb{R})}\sim \left\|u\right\|^2_{H^\mu(\Omega)}.
\end{eqnarray}
Then by Holder's inequality and  (\ref{normeq1}), it follows that
\begin{eqnarray*}
\left\|u\right\|^2_{H^\mu(\Omega)}\sim B( {u},{u})\le \left\|{}_{a}\mathbb{D}_x^{\mu,\lambda}{u}\right\|_{L^2(\Omega)}\left\|{}_{x}\mathbb{D}_b^{\mu,\lambda}{u}\right\|_{L^2(\Omega)}
\stackrel{<}{\sim}\left\|{}_{a}\mathbb{D}_x^{\mu,\lambda}{u}\right\|_{L^2(\Omega)}\left\|u\right\|_{H^\mu(\Omega)};
\end{eqnarray*}
thus
\begin{eqnarray}
\left\|u\right\|_{H^\mu(\Omega)}\sim \left\|{}_{a}\mathbb{D}_x^{\mu,\lambda}{u}\right\|_{L^2(\Omega)}.
\end{eqnarray}
 The proof for $\left\|u\right\|_{H^\mu(\Omega)}\sim \left\|{}_{x}\mathbb{D}_b^{\mu,\lambda}{u}\right\|_{L^2(\Omega)}$ is similar.

Since for $\mu\in (\frac{1}{2},1]$,  $2\mu\theta\in(0,{\pi}]$ and the sign of $\cos(2\mu\theta)$ may change, one can only get that
\begin{eqnarray}
\left({}_{-\infty}\mathbb{D}_x^{\mu,\lambda}{\tilde{u}},{}_{x}\mathbb{D}_\infty^{\mu,\lambda}{\tilde{u}}\right)_{L^2(\mathbb{R})}
<\left\|{}_{-\infty}\mathbb{D}_x^{\mu,\lambda}\tilde{u}\right\|^2_{L^2(\mathbb{R})}\stackrel{<}{\sim}\left\|u\right\|^2_{H^\mu(\Omega)};
\end{eqnarray}
we consider (\ref{eqnarreq2ee}) instead. Starting from (\ref{thereorm2eq1}) and (\ref{thereorm2eq2}), one has
\begin{eqnarray*}
&&2\pi\cdot B(u,u)=-\int_{\mathbb{R}}\big|(\lambda+i\omega)^{\mu}\mathscr{F}[{\tilde{u}}](\omega)\big|^2\cos(2\mu \theta)d\omega
+(1+c_0)\lambda^{2\mu}\int_{\mathbb{R}}\big|\mathscr{F}[{\tilde{u}}](\omega)\big|^2d\omega\\
&&~~=\int_{\mathbb{R}}\left(-\cos(2\mu \theta)+\frac{(1+c_0)\lambda^{2\mu}}{(\lambda^2+|\omega|^2)^{\mu}}\right)\big|(\lambda+i\omega)^{\mu}\mathscr{F}[{\tilde{u}}](\omega)\big|^2d\omega\\
&&~~=\int_{\mathbb{R}}\Big(-\cos(2\mu \theta)+(1+c_0)\cos^{2\mu}(\theta)\Big)\big|(\lambda+i\omega)^{\mu}\mathscr{F}[{\tilde{u}}](\omega)\big|^2d\omega,
\end{eqnarray*}
where  $\tan(\theta)=\frac{|\omega|}{\lambda}$ has been used in the last step.
Letting
\begin{eqnarray}
g(\hat{\theta})=-\cos(2\mu \hat{\theta})+\cos^{2\mu}(\hat{\theta})\quad \hat{\theta}\in \big[0,\frac{\pi}{2}\big],
\end{eqnarray}
then
\begin{eqnarray}
g^{\prime}(\hat{\theta})=2\mu\left(\sin(2\mu\hat{\theta})-\cos^{2\mu-1}(\hat{\theta})\sin(\hat{\theta})\right).
\end{eqnarray}
By Lemma \ref{lemmaeq2}, $g(\hat{\theta})$ is strictly increasing in $[0,\frac{\pi}{2}]$, so $g(\frac{\pi}{2})\ge g(\hat{\theta})\ge g(0)=0$; and $g(\hat{\theta})>0$ everywhere except that $\hat{\theta}=0$.
Obviously,
\begin{eqnarray*}
-\cos(2\mu \hat{\theta})+(1+c_0)\cos^{2\mu}(\hat{\theta})=g(\hat{\theta})+c_0\cos^{2\mu}(\hat{\theta}).
\end{eqnarray*}
For any given $\theta_0\in (0,\frac{\pi}{2})$, it follows that
\begin{eqnarray*}
 \min\left\{g(\theta_0),c_0\cos^{2\mu}(\theta_0)\right\}\le g(\hat{\theta})+c_0\cos^{2\mu}(\hat{\theta})\le g(\frac{\pi}{2})+c_0 \quad \hat{\theta}\in \big[0,\frac{\pi}{2}\big],
\end{eqnarray*}
where the property that $\cos^{2\mu}(\theta)\,(\ge0)$ is strictly decreasing in $[0,\frac{\pi}{2}]$ is used. Therefore,
\begin{eqnarray}\label{Theoremeq3}
B( {u},{u})\sim\left\|{}_{-\infty}\mathbb{D}_x^{\mu,\lambda}\tilde{u}\right\|^2_{L^2(\mathbb{R})}\sim \left\|\tilde{u}\right\|^2_{H^\mu(\mathbb{R})}\sim \left\|u\right\|^2_{H^\mu(\Omega)}.
\end{eqnarray}
Assume that (the proof will be given in  Lemma \ref{lemmaeq4})
\begin{eqnarray}\label{assumeq1}
\left\|u\right\|_{L^2(\Omega)}\stackrel{<}{\sim}\left\|{}_{a}\mathbb{D}_x^{\mu,\lambda}{u}\right\|_{L^2(\Omega)} {~~\rm and~~}\left\|u\right\|_{L^2(\Omega)}\stackrel{<}{\sim}\left\|{}_{x}\mathbb{D}_b^{\mu,\lambda}{u}\right\|_{L^2(\Omega)}.
\end{eqnarray}
Combining Lemma \ref{lemmaeq3}, (\ref{normeq1}), (\ref{assumeq1}), and
\begin{equation}\label{Theoremeq4}
 B( {u},{u})\le \left\|{}_{a}\mathbb{D}_x^{\mu,\lambda}{u}\right\|_{L^2(\Omega)}\left\|{}_{x}\mathbb{D}_b^{\mu,\lambda}{u}\right\|_{L^2(\Omega)}
 +(1+c_0)\lambda^{2\mu}\|u\|^2_{L^2(\Omega)},
\end{equation}
it follows that
\begin{eqnarray*}
 B( {u},{u})\stackrel{<}{\sim} \left\|{}_{a}\mathbb{D}_x^{\mu,\lambda}{u}\right\|_{L^2(\Omega)}\left\|u\right\|_{H^\mu(\Omega)}
 +(1+c_0)\lambda^{2\mu}\left\|{}_{a}\mathbb{D}_x^{\mu,\lambda}{u}\right\|_{L^2(\Omega)}\left\|u\right\|_{H^\mu(\Omega)}
\stackrel{<}{\sim}\left\|{}_{a}\mathbb{D}_x^{\mu,\lambda}{u}\right\|_{L^2(\Omega)}\left\|u\right\|_{H^\mu(\Omega)}.
\end{eqnarray*}
Using (\ref{Theoremeq3}) and (\ref{normeq1}) again, one has
\begin{eqnarray}
\left\|u\right\|_{H^\mu(\Omega)}\sim \left\|{}_{a}\mathbb{D}_x^{\mu,\lambda}{u}\right\|_{L^2(\Omega)}.
\end{eqnarray}
The proof for ${}_{x}\mathbb{D}_b^{\mu,\lambda}{u}$ is similar.

\end{proof}

For $u(x)\in L^2(\Omega)$, let $\tilde{u}(x)$ and $\tilde{v}(x)$  be the zero extensions of $u(x)$ and $e^{-\lambda x}x^{\mu-1}$ from $\Omega$ and $(0,b-a]$, respectively, to $\mathbb{R}$.
Define $\tilde{v}*\tilde{u}=\int_{\mathbb{R}}\tilde{v}(x-\xi)\tilde{u}(\xi)d\xi$. Then ${}_a\mathbb{D}_x^{-\mu,\lambda}u(x)=\frac{\tilde{v}*\tilde{u}}{\Gamma(\mu)}\big|_{\Omega}$. Using Young's inequality \cite[p. 90, Theorem 4.30]{Adams:75},
one has
\begin{eqnarray}\label{lefttermpe1}
\left\|{}_a\mathbb{D}_x^{-\mu,\lambda}u\right\|_{L^2(\Omega)}\le  \left\|\tilde{v}*\tilde{u}\right\|_{L^2(\mathbb{R})}\le \frac{1}{\Gamma(\mu)}\left\|e^{-\lambda x}x^{\mu-1}\right\|_{L_1(0,b-a)}\left\|u\right\|_{L^2(\Omega)}\le \frac{(b-a)^{\mu}}{\Gamma(\mu+1)}\left\|u\right\|_{L^2(\Omega)}.
\end{eqnarray}

For the right tempered fractional integral, it follows that
\begin{eqnarray*}
&&\left\|{}_x\mathbb{D}_b^{-\mu,\lambda}u\right\|^2_{L^2(\Omega)}=\int_a^b\left(\int_x^b\frac{(\xi-x)^{\mu-1}}{\Gamma(\mu)}e^{-\lambda(\xi-x)}u(\xi)d\xi\right)^2dx\nonumber\\
&&~~=\int_a^b\left(\int_{-b}^{-x}\frac{(-\xi-x)^{\mu-1}}{\Gamma(\mu)}e^{-\lambda(-\xi-x)}u(-\xi)d\xi\right)^2dx\nonumber\\
&&~~=\int_{-b}^{-a}\left(\int_{-b}^{x}\frac{(x-\xi)^{\mu-1}}{\Gamma(\mu)}e^{-\lambda(x-\xi)}u(-\xi)d\xi\right)^2dx.\nonumber
\end{eqnarray*}
Let $h(x)=u(-x)$ and $\tilde{\Omega}=(-b,-a)$. By (\ref{lefttermpe1}), one has
\begin{eqnarray}\label{righttermpe2}
&&\left\|{}_x\mathbb{D}_b^{-\mu,\lambda}u\right\|^2_{L^2(\Omega)}=\left\|{}_{-b}\mathbb{D}_x^{-\mu,\lambda}h\right\|^2_{L^2(\tilde{\Omega})}\le\frac{(b-a)^{\mu}}{\Gamma(\mu+1)}\left\|h\right\|_{L^2(\tilde{\Omega})}= \frac{(b-a)^{\mu}}{\Gamma(\mu+1)}\left\|u\right\|_{L^2(\Omega)}.
\end{eqnarray}

\begin{lemma}\label{lemmaeq4}
Let $u\in H_0^{\mu}(\Omega)$, and $ \lambda,\,\mu\ge0$. Then
\begin{eqnarray}
\left\|u\right\|_{L^2(\Omega)}\stackrel{<}{\sim}\left\|{}_{a}\mathbb{D}_x^{\mu,\lambda}{u}\right\|_{L^2(\Omega)} {~~\rm and~~}\left\|u\right\|_{L^2(\Omega)}\stackrel{<}{\sim}\left\|{}_{x}\mathbb{D}_b^{\mu,\lambda}{u}\right\|_{L^2(\Omega)}.
\end{eqnarray}
\end{lemma}
\begin{proof}
For $u\in H_0^{\mu}(\Omega)$, by Theorem \ref{theorem11}, one has ${}_a\mathbb{D}_x^{\mu,\lambda}u \in L^2(\Omega)$ and ${}_x\mathbb{D}_b^{\mu,\lambda}u\in L^2(\Omega)$.
Since $H^{\mu}_0(\Omega)$ is the closure of $C_0^{\infty}(\Omega)$ w.r.t. $\|\cdot\|_{H^\mu(\Omega)}$,  there exists  a sequence $u_n\in C_0^{\infty}(\Omega)$ such that $\lim\limits_{n\to \infty}u_n=u$, and using (\ref{propertyeq3}), it follows that
\begin{eqnarray}\label{Lemma4eq1}
{}_a\mathbb{D}_x^{-\mu,\lambda}\,{}_a\mathbb{D}_x^{\mu,\lambda}u_n={}_x\mathbb{D}_b^{-\mu,\lambda}\,{}_x\mathbb{D}_b^{\mu,\lambda}u_n=u_n.
\end{eqnarray}
Then combining (\ref{Lemma4eq1}), (\ref{lefttermpe1}) and (\ref{normeq1}), one has
\begin{eqnarray}\label{eeetqq}
&&~\left\|{}_a\mathbb{D}_x^{-\mu,\lambda}\,{}_a\mathbb{D}_x^{\mu,\lambda}u-u\right\|_{L^2(\Omega)}
\le\left\|{}_a\mathbb{D}_x^{-\mu,\lambda}\,{}_a\mathbb{D}_x^{\mu,\lambda}\left(u-u_n\right)\right\|_{L^2(\Omega)}
+\left\|u_n-u\right\|_{L^2(\Omega)}\nonumber\\
&&\stackrel{<}{\sim}\left\|{}_a\mathbb{D}_x^{\mu,\lambda}\left(u-u_n\right)\right\|_{L^2(\Omega)}+\left\|u_n-u\right\|_{L^2(\Omega)}
\stackrel{<}{\sim}\left\|u-u_n\right\|_{H^\mu(\Omega)}+\left\|u-u_n\right\|_{L^2(\Omega)}\to 0.\nonumber
\end{eqnarray}
The proof for the right tempered case is similar. Then, using (\ref{lefttermpe1}) and (\ref{righttermpe2}), it yields
\begin{eqnarray*}
&&\left\|u\right\|_{L^2(\Omega)}=\left\|{}_a\mathbb{D}_x^{-\mu,\lambda}\,{}_a\mathbb{D}_x^{\mu,\lambda}u\right\|_{L^2(\Omega)}
\stackrel{<}{\sim}\left\|{}_{a}\mathbb{D}_x^{\mu,\lambda}{u}\right\|_{L^2(\Omega)},\\
&&\left\|u\right\|_{L^2(\Omega)}=\left\|{}_x\mathbb{D}_b^{-\mu,\lambda}\,{}_x\mathbb{D}_b^{\mu,\lambda}u\right\|_{L^2(\Omega)}
\stackrel{<}{\sim}\left\|{}_{x}\mathbb{D}_b^{\mu,\lambda}{u}\right\|_{L^2(\Omega)},
\end{eqnarray*}
and one ends the proof.
\end{proof}
\begin{corollary}\label{corollaryeq35}
 If $u\in H_0^{\mu}(\Omega)$ with $\mu\in (0,1],\mu\not=\frac{1}{2}$,  by Lemma \ref{lemmaeq3}, for all $\lambda\ge 0$, one has
\begin{eqnarray}
\left|u\right|^2_{H^\mu(\Omega)}\sim\left\|u\right\|^2_{H^\mu(\Omega)}\sim \left\|{}_{a}\mathbb{D}_x^{\mu,\lambda}{u}\right\|^2_{L^2(\Omega)}\sim \left\|{}_{x}\mathbb{D}_b^{\mu,\lambda}{u}\right\|^2_{L^2(\Omega)};
\end{eqnarray}
and if $0<\mu_1<\mu$ and $\mu_1\not=\frac{1}{2}$, it holds that
\begin{equation}
\left\|{}_{a}\mathbb{D}_x^{\mu_1,\lambda}{u}\right\|^2_{L^2(\Omega)} \stackrel{<}{\sim} \left\|{}_{a}\mathbb{D}_x^{\mu,\lambda}{u}\right\|^2_{L^2(\Omega)}
{~\rm and~}\left\|{}_{x}\mathbb{D}_b^{\mu_1,\lambda}{u}\right\|^2_{L^2(\Omega)} \stackrel{<}{\sim} \left\|{}_{x}\mathbb{D}_b^{\mu,\lambda}{u}\right\|^2_{L^2(\Omega)}.
\end{equation}
And all the conclusions  apply to  $u\in H^{\mu}(\Omega)$ directly for $0<\mu<\frac{1}{2}$
 (in fact,  $H^{\mu}(\Omega)=H_0^{\mu}(\Omega)$ for $0\le \mu\le\frac{1}{2}$; see, e.g., \cite[Chapter 33]{Tartar:07}).
\end{corollary}

 When developing the method of operator splitting for space fractional  problems \cite{Deng:13,Qiu:15,Xu:14} or carrying on the  theory analysis involving time fractional derivatives, one also needs the variational properties of fractional integrals.
Note that the tempered Caputo fractional derivative always has the form
 ${}_a^C\mathbb{D}_t^{\mu,\lambda}u(x)={}_a\mathbb{D}_x^{-(n-\mu),\lambda}\,{\rm D}_x^n u(x)\,\,(n-1<\mu <n, n\in \mathbb{N^+})$;
 and for $1<\alpha<2$ with $u(a)=0$,  similar to the standard  Riemann-Liouville derivative (see (\cite{Deng:13,Qiu:15}),
  the tempered Riemann-Liouville derivatives have the splitting forms
\begin{eqnarray*}
&&{}_a D_x^{\alpha,\lambda}u(x)=\left(\lambda+\frac{d}{dx}\right){}_a\mathbb{D}_x^{-(2-\alpha),\lambda}\left(\lambda+\frac{d}{dx}\right)u(x)
-\alpha\lambda^{\alpha-1}\left(\lambda+\frac{d}{dx}\right)u(x)+\lambda^\alpha(\alpha-1)u(x),
\end{eqnarray*}
and
\begin{eqnarray*}
&&{}_x D_b^{\alpha,\lambda}u(x)=\left(\lambda-\frac{d}{dx}\right){}_x\mathbb{D}_b^{-(2-\alpha),\lambda}\left(\lambda-\frac{d}{dx}\right)u(x)
-\alpha\lambda^{\alpha-1}\left(\lambda-\frac{d}{dx}\right)u(x)+\lambda^\alpha(\alpha-1)u(x),
\end{eqnarray*}
respectively, where the properties (\ref{equationee1}) and (\ref{equationee2}) are used. Therefore, we will limit our discussions for the tempered fractional integrals to the case: $0<\mu<1$.


\begin{theorem} \label{Theoremeq23}
Let the real function $u\in L^2(\Omega)$ and $\mu\in (0,1)$. Then
\begin{eqnarray}
\left\|{}_a\mathbb{D}_x^{-\frac{\mu}{2},\lambda}u\right\|^2_{L^2(\Omega)}\sim\left({}_a\mathbb{D}_x^{-\mu,\lambda}u,u\right)=\left(u, {}_x\mathbb{D}_b^{-\mu,\lambda}u\right)
\sim \left\|{}_x\mathbb{D}_b^{-\frac{\mu}{2},\lambda}u\right\|^2_{L^2(\Omega)}.
\end{eqnarray}
\end{theorem}
\begin{proof}
First, let $\tilde{u}$ be the zero extension of $u$ from $\Omega$ to $\mathbb{R}$; using Property \ref{propertyeq2} and the Plancherel theorem, one has
\begin{eqnarray}\label{integralnorm1}
&&\left(u, {}_x\mathbb{D}_b^{-\mu,\lambda}u\right)=\frac{1}{2\pi}\int_{\mathbb{R}}\mathscr{F}[\tilde{u}](\omega)
\overline{\left(\lambda-i\omega\right)^{-\mu}\mathscr{F}[\tilde{u}](\omega)} d\omega\\
&&=\frac{1}{2\pi}\int_{\mathbb{R}}\cos(\mu \theta)\left|\lambda-i\omega\right|^{-\mu}\mathscr{F}[\tilde{u}](\omega)\overline{\mathscr{F}[\tilde{u}](\omega)}d\omega\nonumber\\
&&\ge\frac{\cos(\frac{\pi\mu}{2})}{2\pi}\int_{\mathbb{R}}\left|\mathscr{F}[{}_x\mathbb{D}_{\infty}^{-\frac{\mu}{2},\lambda}\tilde{u}](\omega)\right|^2d\omega\nonumber\\
&&\ge\frac{\cos(\frac{\pi\mu}{2})}{2\pi}\int_{\Omega}\left|{}_x\mathbb{D}_{b}^{-\frac{\mu}{2},\lambda}u\right|^2dx
=\frac{\cos(\frac{\pi\mu}{2})}{2\pi}\left\|{}_x\mathbb{D}_b^{-\frac{\mu}{2},\lambda}u\right\|^2_{L^2(\Omega)},\nonumber
\end{eqnarray}
where $\theta$ is given in (\ref{thetaeq}) and in the second step, $\overline{\mathscr{F}(\tilde{u})(-\omega)}=\mathscr{F}(\tilde{u})(\omega)$ is used.
Since $\left|\lambda-i\omega\right|^{-\mu}=\left|\lambda+i\omega\right|^{-\mu}$, starting from the second step of (\ref{integralnorm1}), one also has
\begin{eqnarray}\label{integralnorm2}
&&\left(u,{}_x\mathbb{D}_b^{-\mu,\lambda}u\right)=\frac{1}{2\pi}\int_{\mathbb{R}}\cos(\mu \theta)\left|\lambda+i\omega\right|^{-\mu}\mathscr{F}[\tilde{u}](\omega)\overline{\mathscr{F}[\tilde{u}](\omega)}d\omega\nonumber\\
&&\ge\frac{\cos(\frac{\pi\mu}{2})}{2\pi}\int_{\mathbb{R}}\left|\mathscr{F}[{}_{-\infty}\mathbb{D}_{x}^{-\frac{\mu}{2},\lambda}\tilde{u}](\omega)\right|^2d\omega\\
&&\ge\frac{\cos(\frac{\pi\mu}{2})}{2\pi}\int_{\Omega}\left|{}_a\mathbb{D}_{x}^{-\frac{\mu}{2},\lambda}u\right|^2dx
=\frac{\cos(\frac{\pi\mu}{2})}{2\pi}\left\|{}_a\mathbb{D}_x^{-\frac{\mu}{2},\lambda}u\right\|^2_{L^2(\Omega)}.\nonumber
\end{eqnarray}
From  (\ref{semigroup}) and (\ref{adjonteq}), it follows that
\begin{eqnarray}\label{integralnorm3}
\left(u, {}_x\mathbb{D}_b^{-\mu,\lambda}u\right)=\left(u,\, {}_x\mathbb{D}_b^{-\frac{\mu}{2},\lambda}{}_x\mathbb{D}_b^{-\frac{\mu}{2},\lambda}u\right)
=\left({}_a\mathbb{D}_x^{-\frac{\mu}{2},\lambda}u,\, {}_x\mathbb{D}_b^{-\frac{\mu}{2},\lambda}u\right).
\end{eqnarray}
Note that
\begin{eqnarray} \label{integralnorm4}
\left({}_a\mathbb{D}_x^{-\frac{\mu}{2},\lambda}u,\, {}_x\mathbb{D}_b^{-\frac{\mu}{2},\lambda}u\right)
\le\left\|{}_a\mathbb{D}_x^{-\frac{\mu}{2},\lambda}u\right\|_{L^2(\Omega)}\left\|{}_x\mathbb{D}_b^{-\frac{\mu}{2},\lambda}u\right\|_{L^2(\Omega)}.
\end{eqnarray}
Then combining (\ref{integralnorm1})--(\ref{integralnorm4}), one obtains
\begin{equation}
\left\|{}_a\mathbb{D}_x^{-\frac{\mu}{2},\lambda}u\right\|^2_{L^2(\Omega)}\sim\left\|{}_x\mathbb{D}_b^{-\frac{\mu}{2},\lambda}u\right\|^2_{L^2(\Omega)}
\sim\left({}_a\mathbb{D}_x^{-\frac{\mu}{2}}u,\, {}_x\mathbb{D}_b^{-\frac{\mu}{2},\lambda}u\right).
\end{equation}
The proof is completed.
\end{proof}

\begin{theorem} [Embeddedness] \label{fractionalembongin}
 For $0\le \mu_1<\mu_2$, and  $u\in L^2(\Omega)$, it follows that
 \begin{eqnarray}
 &&\left\|{}_a\mathbb{D}_x^{-{\mu_2},\lambda}u\right\|_{L^2(\Omega)}\stackrel{<}{\sim}\left\|{}_a\mathbb{D}_x^{-{\mu_1},\lambda}u\right\|_{L^2(\Omega)},\label{embodingeq1}\\
 &&\left\|{}_x\mathbb{D}_b^{-{\mu_2},\lambda}u\right\|_{L^2(\Omega)}\stackrel{<}{\sim}\left\|{}_x\mathbb{D}_b^{-{\mu_1},\lambda}u\right\|_{L^2(\Omega)}.\label{embodingeq2}
 \end{eqnarray}
 \end{theorem}
 \begin{proof}
 Since
 \begin{eqnarray*}
 &&{}_a\mathbb{D}_x^{-{\mu_2},\lambda}u={}_a\mathbb{D}_x^{-({\mu_2-\mu_1}),\lambda}{}_a\mathbb{D}_x^{-{\mu_1},\lambda}u,\\
  &&{}_x\mathbb{D}_b^{-{\mu_2},\lambda}u={}_x\mathbb{D}_b^{-({\mu_2-\mu_1}),\lambda}{}_x\mathbb{D}_b^{-{\mu_1},\lambda}u.
 \end{eqnarray*}
 Similar to (\ref{lefttermpe1}) and (\ref{righttermpe2}), one has
 \begin{eqnarray*}
 &&\left\|{}_a\mathbb{D}_x^{-{\mu_2},\lambda}u\right\|_{L^2(\Omega)}\le \frac{(b-a)^{\mu_2-\mu_1}}{\Gamma(\mu_2-\mu_1+1)}\left\|{}_a\mathbb{D}_x^{-{\mu_1},\lambda}u\right\|_{L^2(\Omega)},\\
  &&\left\|{}_x\mathbb{D}_b^{-{\mu_2},\lambda}u\right\|_{L^2(\Omega)}\le \frac{(b-a)^{\mu_2-\mu_1}}{\Gamma(\mu_2-\mu_1+1)}\left\|{}_x\mathbb{D}_b^{-{\mu_1},\lambda}u\right\|_{L^2(\Omega)}.
\end{eqnarray*}
 Then we complete the proof.
 \end{proof}
\section {Numerical analysis and implementation for the space tempered fractional stationary equation and time tempered fractional equation
 }
 Now, we apply the above provided theoretical framework to solve the models involving the tempered fractional calculus. The strict numerical analysis and efficient implementation are detailedly discussed.
  First, we consider the space tempered fractional equation.
\subsection {Space tempered fractional advection dispersion model}
For the convenience of presentation, we discuss a simple space tempered fractional stationary model, but it can be easily extended to the corresponding time evolution or high  dimension  problem \cite{Baeumera:10,Cartea:07} and \cite{Deng:08,Zhaoyi:15}. The model is given by
\begin{eqnarray}\label{model1}
-(1-p)\,{}_a D_x^{\alpha,\lambda}u(x)-p\,{}_x D_b^{\alpha,\lambda}u(x)+m(x)u^{\prime}+c(x)u=f(x)
\end{eqnarray}
with $u(a)=u(b)=0$  and $u^{\prime}=\frac{du}{dx}$, where $1<\alpha\le2$, $ 0\le p\le1$, $m(x)\in C^1(\overline{\Omega})$, $ c(x)\in C(\overline{\Omega})$, and $c(x)-\frac{1}{2}m^{\prime}>0$.
\subsubsection{Model analysis}
Now consider the Galerkin weak formulation of model (\ref{model1}). For $f\in H^{-\frac{\alpha}{2}}(\Omega)$, find $u\in H^{\mu}_0(\Omega)$ such that
\begin{eqnarray}\label{weakform}
A\left(u,v\right)=\left\langle f,v\right\rangle \quad\forall v\in H_0^{\frac{\alpha}{2}}(\Omega),
\end{eqnarray}
where
\begin{eqnarray}
&&A\left(u,v\right)=-(1-p)\left({}_{a}\mathbb{D}_x^{\frac{\alpha}{2},\lambda}{u},\,{}_{x}\mathbb{D}_b^{\frac{\alpha}{2},\lambda}{v}\right)
-p\left({}_{x}\mathbb{D}_b^{\frac{\alpha}{2},\lambda}{u},\,{}_{a}\mathbb{D}_x^{\frac{\alpha}{2},\lambda}{v}\right)+\left(1-p\right)\lambda^{\alpha}\left(u,\,v\right)\\
&&~+p\lambda^{\alpha}\left(u,\,v\right)
+\left(\alpha(1-2p)\lambda^{\alpha-1}\right)\left({}_{a}\mathbb{D}_x^{\frac{\alpha}{2},0}{u},\,{}_{x}\mathbb{D}_b^{1-\frac{\alpha}{2},0}{v}\right)
+\left({}_{a}\mathbb{D}_x^{\frac{\alpha}{2},0}{u},\,{}_{x}\mathbb{D}_b^{1-\frac{\alpha}{2},0}({mv})\right)+\left(cu,v\right). \nonumber
\end{eqnarray}

In fact, assuming that $u$ is smooth enough and $v\in C_0^\infty(\Omega)$, one has
\begin{eqnarray}\label{derivatiequ1}
&&\big({}_{a}\mathbb{D}_x^{\alpha,\lambda}{u},v\big)=\big({}_aD_x^{\alpha}(e^{\lambda x}u),\,e^{-\lambda x}v\big)\\
&&=-\left({}_aD_x^{-(2-\alpha)}\frac{d}{dx}(e^{\lambda x}u),\,\frac{d}{dx}\left(e^{-\lambda x}v\right)\right)\nonumber\\
&&=-\left({}_aD_x^{-(1-\frac{\alpha}{2})}\frac{d}{dx}(e^{\lambda x}u),\,{}_xD_b^{-(1-\frac{\alpha}{2})}\frac{d}{dx}\left(e^{-\lambda x}v\right)\right)\nonumber\\
&&=\left({}_aD_x^{\frac{\alpha}{2}}(e^{\lambda x}u),\,{}_xD_b^{\frac{\alpha}{2}}\left(e^{-\lambda x}v\right)\right)=\left({}_{a}\mathbb{D}_x^{\frac{\alpha}{2},\lambda}{u},\,{}_{x}\mathbb{D}_b^{\frac{\alpha}{2},\lambda}{v}\right).\nonumber
\end{eqnarray}
The derivation process of $\left({}_{x}\mathbb{D}_b^{\frac{\alpha}{2},\lambda}{u},\,{}_{a}\mathbb{D}_x^{\mu,\lambda}{v}\right)$ is similar. The results involving the first derivative can be obtained by  using the following Lemma  with $\lambda=0$, $\gamma=1$, and $q=\frac{\alpha}{2}$.

\begin{lemma} \label{section4lemma3}
For all $ \lambda\ge 0$, $0\le q<\gamma\le1$, let $u\in C^1(\Omega)$, $u(a)=0$, and $v\in C_0^{\infty}(\Omega)$. Then
\begin{eqnarray}
\left({}_{a}\mathbb{D}_x^{\gamma,\lambda}{u}, v\right)=\left({}_{a}\mathbb{D}_x^{q,\lambda}{u},{}_{x}\mathbb{D}_b^{\gamma-q,\lambda}{v}\right).
\end{eqnarray}
\end{lemma}
\begin{proof} Combining the fact \cite[p. 74]{Podlubny:99} that ${}_aD_x^{q_1+q_2}u(x)= {}_aD_x^{q_1}{}_aD_x^{q_2}u(x)\,\,{\rm for~} 0\le q_1,\,q_2\le 1$, $ u(a)=0$, and  Definition \ref{definition:2}, one has
\begin{eqnarray}
{}_{a}\mathbb{D}_x^{\gamma,\lambda}{u}={}_{a}\mathbb{D}_x^{\gamma-q,\lambda}{}_{a}\mathbb{D}_x^{q,\lambda}u.
\end{eqnarray}
And by (\ref{equationee1}) and (\ref{lefttermpe1}), for any $ 0\le \mu_1\le 1$, it holds that
\begin{eqnarray*}
\left\|{}_{a}\mathbb{D}_x^{\mu_1,\lambda}{u}\right\|_{L^2(\Omega)}=\left\|{}_{a}^C\mathbb{D}_x^{\mu_1,\lambda}{u}\right\|_{L^2(\Omega)}
=\left\|{}_{a}\mathbb{D}_x^{-(1-\mu_1),\lambda}\,{\rm D}_x^1{u}\right\|_{L^2(\Omega)}\stackrel{<}{\sim}\left\| {\rm D}_x^1{u}\right\|_{L^2(\Omega)}.
\end{eqnarray*}
Then
\begin{eqnarray}
\left({}_{a}\mathbb{D}_x^{\gamma,\lambda}{u}, v\right)=\left({}_{a}\mathbb{D}_x^{\gamma-q,\lambda}{}_{a}\mathbb{D}_x^{q,\lambda}u, v\right)=\left({}_{a}\mathbb{D}_x^{q,\lambda}{u},{}_{x}\mathbb{D}_b^{\gamma-q,\lambda}{v}\right)
\end{eqnarray}
follows after using the result given in  \cite[Lemma 2.4]{Li:10}.
\end{proof}
\begin{theorem}\label{eqpplittheorm}
For all $1<\alpha\le2$,  problem (\ref {weakform}) has an unique solution $u\in H_0^{\frac{\alpha}{2}}(\Omega)$, and it holds that
\begin{equation}\label{theoremeq23}
\left\|u\right\|_{H^\frac{\alpha}{2}(\Omega)}\stackrel{<}{\sim} \|f\|_{H^{-\frac{\alpha}{2}}(\Omega)}.
\end{equation}
\end{theorem}
\begin{proof} First,  by Lemma \ref{lemmaeq3} and Corollary \ref{corollaryeq35},  it follows that
\begin{eqnarray}
&&\left|A\left(u,v\right)\right|\le(1-p)\left\|u\right\|_{H^\frac{\alpha}{2}(\Omega)}\left\|v\right\|_{H^\frac{\alpha}{2}(\Omega)}
+\left(1-p\right)\lambda^{\alpha}\left\|u\right\|_{L^2(\Omega)}\left\|v\right\|_{L^2(\Omega)}\\
&&+p\left\|v\right\|_{H^\frac{\alpha}{2}(\Omega)}\left\|u\right\|_{H^\frac{\alpha}{2}(\Omega)}+p\lambda^{\alpha}\left\|u\right\|_{L^2(\Omega)}\left\|u\right\|_{L^2(\Omega)}\nonumber\\
&&+\left|\alpha(1-2p)\lambda^{\alpha-1}\right|\left\|u\right\|_{H^\frac{\alpha}{2}(\Omega)}\left\|v\right\|_{H^\frac{\alpha}{2}(\Omega)}
+\left\|u\right\|_{H^\frac{\alpha}{2}(\Omega)}\left\|mv\right\|_{H^\frac{\alpha}{2}(\Omega)}+\left\|c\right\|_{\infty}\left\|u\right\|_{L^2(\Omega)}\left\|v\right\|_{L^2(\Omega)}. \nonumber\\
&&\stackrel{<}{\sim}\left\|u\right\|_{H^\frac{\alpha}{2}(\Omega)}\left\|v\right\|_{H^\frac{\alpha}{2}(\Omega)}\nonumber,
\end{eqnarray}
where $mv\in H_0^{\frac{\alpha}{2}}(\Omega)$ and $\left\|mv\right\|_{H^\frac{\alpha}{2}(\Omega)}\stackrel{<}{\sim} \left\|v\right\|_{H^\frac{\alpha}{2}(\Omega)}$ \cite[Lemma 3.2]{Ervin:05} have been used.

Secondly, assuming that $u\in C_0^{\infty}(\Omega)$, by Lemma \ref{section4lemma3} and using integration by parts it follows that
\begin{eqnarray}
&&\alpha(1-2p)\lambda^{\alpha-1}\left({}_{a}\mathbb{D}_x^{\frac{\alpha}{2},0}{u},\,{}_{x}\mathbb{D}_b^{1-\frac{\alpha}{2},0}{u}\right)
+\left({}_{a}\mathbb{D}_x^{\frac{\alpha}{2},0}{u},\,{}_{x}\mathbb{D}_b^{1-\frac{\alpha}{2},0}({mu})\right)=\left(m u^{\prime},u\right)=-\frac{1}{2}\left({m}^{\prime}u,u\right).\nonumber
\end{eqnarray}
 Then
\begin{eqnarray}
A\left(u,u\right)&=&-\left({}_{a}\mathbb{D}_x^{\frac{\alpha}{2},\lambda}{u},\,{}_{x}\mathbb{D}_b^{\frac{\alpha}{2},\lambda}{u}\right)
+(1+c_0)\lambda^{\alpha}\left(u,\,u\right)+\left(\left(c-{m^{\prime}}/{2}-c_0\lambda^{\alpha}\right)u,u\right). \nonumber
\end{eqnarray}
If $\lambda>0$, we take $c_0=\inf_{x\in {\Omega}}\lambda^{-\alpha}\left(c(x)-{m^{\prime}}/{2}\right)>0$. By Theorem \ref{theorem2}, one always has
\begin{equation}\label{eqnarreq}
\left\|u\right\|^2_{H^\frac{\alpha}{2}(\Omega)}\stackrel{<}{\sim}A\left(u,u\right).
\end{equation}
Noting the density of $C_0^{\infty}(\Omega)$, the continuities of $A(\cdot,\cdot)$ and $\|\cdot\|_{H^\frac{\alpha}{2}(\Omega)}$,  for any $u\in H_0^{\frac{\alpha}{2}}(\Omega)$, one still has (\ref{eqnarreq}).  Then using the Lax-Milgram theorem leads to the desired result.
\end{proof}
\begin{remark}
It seems  that the coercive condition becomes a little bit stronger for $\lambda>0$ than $\lambda=0$, but this further  requirement  might be removed by the fact that it is only the point $\theta=0\,( i.e., \omega=0)$ that destroys the positive lower bound of  $g(\theta)$ and one only needs $\left\|{}_{a}\mathbb{D}_x^{\mu,\lambda}\tilde{u}\right\|^2_{L^2(\Omega)}\stackrel{<}{\sim}B(u,u)$ rather than $\left\|{}_{-\infty}\mathbb{D}_x^{\mu,\lambda}\tilde{u}\right\|^2_{L^2(\mathbb{R})}\stackrel{<}{\sim}B(u,u)$ (see the proof of Theorem \ref{theorem2}). One can easily check this for $\alpha=2$:
 \begin{eqnarray}
 &&-\left({}_{a}\mathbb{D}_x^{1,\lambda}{u},\,{}_{x}\mathbb{D}_b^{1,\lambda}{u}\right)+\lambda^2\left(u, u\right)
 =\left(\left(\frac{d}{dx}+\lambda\right)u,\left(\frac{d}{dx}-\lambda\right)u\right)+\lambda^2\left(u, u\right)=\left|u\right|^2_{H^1(\Omega)}.\nonumber
\end{eqnarray}
\end{remark}

Here, we further discuss the Petrov-Galerkin method, which is popular for the problem with only the one-sided tempered fractional derivative \cite{Zayernourt:15}. First, we show the following lemma.

\begin{lemma}\label{onesidelemma}
If $0<{\alpha}\le 2$ and $\alpha\not=1$, then $u\in H^{\frac{\alpha}{2}}_0(\Omega)$ is equivalent to $e^{\varrho x}u\in H^{\frac{\alpha}{2}}_0(\Omega)$,
i.e., the map $\mathscr{M}:H_0^{\frac{\alpha}{2}}(\Omega)\to H_0^{\frac{\alpha}{2}}(\Omega): u\to e^{\varrho x}u$  is bijection.
\end{lemma}
\begin{proof}
Assume $\varrho\ge0$. Then the conclusion follows by using
\begin{eqnarray}
\left\|{}_x\mathbb{D}_b^{\frac{\alpha}{2},\varrho}(e^{\varrho x}u)\right\|^2= \left\|e^{\varrho x}{}_xD_b^{\frac{\alpha}{2}}u\right\|^2\simeq\left\|{}_xD_b^{\frac{\alpha}{2}}u\right\|^2
\end{eqnarray}
and Corollary \ref{corollaryeq35}. One can similarly prove the case: $\varrho<0$.
\end{proof}

For discussing the one-sided tempered fractional equation, we introduce two bilinear form defined as:
\begin{eqnarray}
&&A_1\left(u, v\right):=\left({}_{a}\mathbb{D}_x^{\frac{\alpha}{2},\lambda}{u},\,{}_{x}\mathbb{D}_b^{\frac{\alpha}{2},\lambda}{v}\right)+
\alpha\lambda^{\alpha-1}\left({}_{a}\mathbb{D}_x^{\frac{\alpha}{2},0}{u},\,{}_{x}\mathbb{D}_b^{1-\frac{\alpha}{2},0}{v}\right)+\lambda^{\alpha}\left(u, v\right),\\
&&A_2\left(u_1, v_1\right):=\left({}_{a}{D}_x^{\frac{\alpha}{2}}{u_1},\,{}_{x}{D}_b^{\frac{\alpha}{2}}{v_1}\right)+
\alpha\lambda^{\alpha-1}\left({}_{a}{D}_x^{\frac{\alpha}{2}}{u_1},\,{}_{x}{D}_b^{1-\frac{\alpha}{2}}{v_1}\right)+(1-\alpha)\lambda^{\alpha}\left(u_1, v_1\right).\nonumber
\end{eqnarray}
Consider the left tempered fractional equation with the homogeneous boundary condition:
\begin{equation}\label{epereeq1}
-{}_a D_x^{\alpha,\lambda}u=-{}_a \mathbb{D}_x^{\alpha,\lambda}u(x)+\lambda^{\alpha}u(x)+\alpha\lambda^{\alpha-1}\frac{du(x)}{d x}=f(x);
\end{equation}
still with the homogeneous boundary condition the companion equation of (\ref{epereeq1}) is:
\begin{equation} \label{epereeq2}
-{}_a {D}_x^{\alpha}u_1(x)+\alpha\lambda^{\alpha-1}\frac{du_1(x)}{d x}+(1-\alpha)\lambda^{\alpha}u_1(x)=e^{\lambda x}f(x),
\end{equation}
where $u_1=e^{\lambda x}u$. The Petrov-Galerkin formulation of (\ref{epereeq1}) is to find $u\in e^{-\lambda x}\cdot H_0^{\frac{\alpha}{2}}(\Omega)$, such that for any $v\in e^{\lambda x}\cdot H_0^{\frac{\alpha}{2}}(\Omega)$,
\begin{equation}\label{potroveq1}
A_1\left(u, v\right)=\left\langle f,v\right\rangle;
\end{equation}
and the Galerkin formulation of the companion equation (\ref{epereeq2}) is to find $u_1\in H_0^{\frac{\alpha}{2}}(\Omega)$, such that for any $v_1\in H_0^{\frac{\alpha}{2}}(\Omega)$,
\begin{eqnarray}  \label{potroveq2}
A_2\left(u_1, v_1\right)=\left\langle e^{\lambda x}f,v_1\right\rangle.
\end{eqnarray}
By Lemma \ref{onesidelemma}, we know that (\ref{potroveq1}) and (\ref{potroveq2}) are equivalent with $u_1=e^{\lambda x}u$.


\begin{lemma}\label{petreqlemma}
For $0\le\lambda^{\alpha}<\frac{\left|\cos(\frac{\pi\alpha}{2})\right|\Gamma^2(\frac{\alpha}{2}+1)}{2\pi(\alpha-1)(b-a)^{\alpha}}$ and $1<\alpha<2$,  the weak formulae (\ref{potroveq1}) and(\ref{potroveq2}) has an unique solution, and which satisfies $\left\|u\right\|_{H^{\frac{\alpha}{2}}(\Omega)}\le \left\|e^{\lambda x}f\right\|_{H^{-\frac{\alpha}{2}}(\Omega)}$.
\end{lemma}
\begin{proof}
For the Galerkin formulation (\ref{potroveq2}),
\begin{eqnarray*}
A_2\left(u_1, u_1\right)&\ge&\frac{\left|\cos( \frac{\pi\alpha}{2})\right|}{2\pi}\left\|{}_a D_x^{\frac{\alpha}{2}}u_1\right\|^2_{L^2(\Omega)}+(1-\alpha)\lambda^{\alpha}\left\|{}_aD^{-\frac{\alpha}{2}}_x\left({}_aD^{\frac{\alpha}{2}}_xu_1\right)\right\|^2_{L^2(\Omega)}\\
&\ge&\left(\frac{\left|\cos( \frac{\pi\alpha}{2})\right|}{2\pi}+(1-\alpha)\lambda^{\alpha}\frac{(b-a)^{\alpha}}{\Gamma^2(\frac{\alpha}{2}+1)}\right)\left\|{}_a D_x^{\frac{\alpha}{2}}u_1\right\|^2_{L^2(\Omega)},
\end{eqnarray*}
where $B(u,u)=\frac{\left|\cos(\pi\mu)\right|}{2\pi}\left\|{}_{-\infty}\mathbb{D}_x^{\mu,\lambda}\tilde{u}\right\|^2_{L^2(\mathbb{R})}\ge \frac{\left|\cos(\pi\mu)\right|}{2\pi}\left\|{}_a\mathbb{D}_x^{\mu,\lambda}{u}\right\|^2_{L^2(\Omega)}$
 and Theorem \ref{fractionalembongin} are used.  For the Petrov-Galerkin formulation (\ref{potroveq1}), taking $u$ and $v$ as $e^{-\lambda x}u_1$ and $e^{\lambda x}u_1$, respectively, leads to $\left|A_1\left(u, v\right)\right|=A_2\left(u_1, u_1\right)$. So, the Bab\v{u}ska Inf-Sup conditions \cite{Babuska:71} are verified.

\end{proof}

In fact, if the model just has the one-sided tempered fractional derivatives ${}_a \mathbb{D}_x^{\alpha_1,\lambda}$ and   ${}_a\mathbb{D}_x^{\alpha_2,\lambda}$ \cite{Zayernourt:15}, it is convenient to convert the Petrov-Galerkin problem into the Galerkin problem, e.g., the tempered fractional advection-diffusion equation
\begin{equation} \label{Petrov_to_Galerkin}
-{}_a \mathbb{D}_x^{\alpha_1,\lambda}u(x)+d\cdot {}_a \mathbb{D}_x^{\alpha_2,\lambda}u(x)=f(x)
\end{equation}
with $u(a)=u(b)=0$, $0\le \alpha_2\le 1$, $1< \alpha_1\le2$, and $f\in H^{-\alpha_1/2}(\Omega)$; by (\ref{derivatiequ1}) and Lemma \ref{section4lemma3}, the Petrov-Galerkin solution of (\ref{Petrov_to_Galerkin}) $u=e^{-\lambda x}u_1$ with $u_1\in H_0^{\alpha_1/2}(\Omega)$ satisfying
\begin{eqnarray}\label{equationeq1}
\qquad-\left({}_{a}{D}_x^{\frac{\alpha_1}{2}}{u_1},\,{}_{x}{D}_b^{\frac{\alpha_1}{2}}{v_1}\right)+d\cdot\left({}_{a}{D}_x^{\frac{\alpha_2}{2}}{u_1},\,{}_{x}{D}_b^{\frac{\alpha_2}{2}}{v_1}\right)=\left\langle e^{\lambda x}f,v_1\right\rangle \,\, v_1\in H_0^{\alpha_1/2}(\Omega).
\end{eqnarray}
 If $ d>-\frac{\left|\cos(\frac{\alpha_1}{2})\right|\Gamma^2(\frac{\alpha_1-\alpha_2}{2}+1)}{2\pi(b-a)^{(\alpha_1-\alpha_2)}}$,
(\ref{equationeq1}) has an unique solution, which follows from Theorem \ref{theorem2} and \ref{fractionalembongin} with $\lambda=0$.
\subsubsection{Numerical implementation}
Now, we discuss the efficient implementation. With the equidistant nodes $a=x_0<x_1\cdots<x_{N}=b$, $i=0,\cdots,N-1$, $x_{i+1}-x_i=h$, the linear element bases are given as
 \begin{eqnarray}
 \phi_{h,k}(x)=\left\{ \begin{array}{ll}\frac{1}{\sqrt{h}}\cdot\frac{x-x_{k}}{h}&x_{k}\le x\le x_{k+1}\\
 \frac{1}{\sqrt{h}}\cdot\frac{x_{k+2}-x}{h}&x_{k+1}<x\le x_{k+2}\end{array}\quad k=0,1,\cdots,N-2.\right.
 \end{eqnarray}
Defining $\phi(x)=\left\{\begin{array}{ll}x&0\le x\le 1\\2-x&1<x\le2\end{array}\right.$, it is easy to check that
\begin{eqnarray}\label{baseeq1}
&&\phi(x)=\phi(2x-1)+\frac{1}{2}\left(\phi(2x)+\phi(2x-2)\right),\label{sacleq1}\\
&&\phi_{h,k}(x)=\frac{1}{\sqrt{h}}\cdot\phi\left(\frac{x-a}{h}-k\right)\quad k=0,1,\cdots,N-2.\end{eqnarray}
Therefore $\phi_{h,k}$ is just the dilation and translation of a single function $\phi_h$. Let $\Phi_h(x)=\left\{\phi_{h,k}(x)\right\}_{k=0}^{N-1}$ and $ S_h={\rm span}\left\{\Phi_h(x)\right\}$. Then the  finite element Galerkin  approximation of (\ref{weakform}) or (\ref{potroveq2}) can be given as: find $u_h\in S_h$, such that
 \begin{eqnarray}
 A\left(u_h,v\right)=\left\langle f,v\right\rangle \quad\forall v\in S_h ~~~ {\rm or} ~~~ A_2\left(u_h,v\right)=\left\langle f,v\right\rangle \quad\forall v\in S_h;
 \end{eqnarray}
and the  Petrov-Galerkin finite element approximation of (\ref{potroveq1}) is: find $u_h\in e^{-\lambda x}\cdot S_h$, such that
\begin{eqnarray}
A_1\left(u_h, v\right)=\left\langle f,v\right\rangle \quad v\in e^{\lambda x}\cdot S_h,
\end{eqnarray}
where the discrete Bab\v{u}ska Inf-sup conditions  \cite{Babuska:71} can be checked similar to the proof of Lemma $\ref{petreqlemma}$.
Both the approximations arrive to the error estimate
\begin{eqnarray}\label{eqarrayeq}
\left\|u-u_h\right\|_{H^{\frac{\alpha}{2}}(\Omega)}\stackrel{<}{\sim} h^{\min(2,s)-\frac{\alpha}{2}}\left\|u\right\|_{H^s(\Omega)}.
\end{eqnarray}

To simplify the calculations and reduce the storage, we present the following result.
\begin{theorem}\label{Theoremnumercialeq}
Let $\mathcal{L}[u](x)=\int_a^xg(x-\xi)u(\xi)d\xi$. Assume that there exist  $\phi_1(x)$ and $\phi_2(x)$ with compact supports ${\rm supp}\phi_1(x)=[0,d_1]$ and $ {\rm supp}\phi_2(x)=[0,d_2]$, such that all the supports of  $\phi_{1,k}(x)=\phi_1(c(x-a)-k)$ and $\phi_{2,k}(x)=\phi_2(c(x-a)-k), k=0,1,\cdots, M,c>0$ lie in $(a,b)$, and define
$\Phi_1(x)=\left\{\phi_{1,k}\right\}_{k=0}^M$ and $\Phi_2(x)=\left\{\phi_{2,k}\right\}_{k=0}^M$.
 Then the matrix $\big(\mathcal{L}[\Phi_1],\Phi_2\big)$ with the element $\big(\mathcal{L}[\Phi_1],\Phi_2\big)_{ij}=\big(\phi_{1,i}(x), \phi_{2,j}(x) \big)$  is Toeplitz. 
\end{theorem}
\begin{proof} Since
\begin{eqnarray}
&&\big(\mathcal{L}[\phi_{1,k_1}],\phi_{2,k_2}\big)=\int_a^b\int_a^xg(x-\xi)\phi_{1,k_1}(\xi)\,d\xi\,\phi_{2,k_2}(x)\,dx\\
&&~~~~~~~=\int_0^{b-a}\int_0^xg(x-\xi)\phi_1(c\xi-k_1)\,d\xi \phi_2(cx-k_2)\,dx\nonumber\\
&&~~~~~~~=\int_{\frac{k_2}{c}}^{\frac{d_2+k_2}{c}}\int_0^xg(x-\xi)\phi_1(c\xi-k_1)\,d\xi \phi_2(cx-k_2)\,dx\nonumber\\
&&~~~~~~~~=\frac{1}{c}\int_0^{d_2}\int_0^{\frac{x+k_2}{c}}g\left(\frac{k_2+x}{c}-\xi\right)\phi_1(c\xi-k_1)\,d\xi \phi_2(x)\,dx\nonumber\\
&&~~~~~~~~=\frac{1}{c^2}\int_0^{d_2}\int_0^{\min\{x+k_2-k_1,d_1\}}g\left(\frac{x+k_2-k_1-\xi}{c}\right)\phi_1(\xi)\,d\xi \phi_2(x)\,dx,\nonumber
\end{eqnarray}
which just depends on the value of $k_2-k_1$. So, the matrix is Toeplitz.
\end{proof}

Because of the property of the dense matrix
\begin{eqnarray}
\left({}_{x}\mathbb{D}_b^{\frac{\alpha}{2},\lambda}{\Phi_h},\,{}_{a}\mathbb{D}_x^{\frac{\alpha}{2},\lambda}{\Phi_h}\right)
=\left({}_{a}\mathbb{D}_x^{\frac{\alpha}{2},\lambda}{\Phi_h},\,{}_{x}\mathbb{D}_b^{\frac{\alpha}{2},\lambda}{\Phi_h}\right)^T,
\end{eqnarray}
we only consider ${\bf A}_l=\left({}_{a}\mathbb{D}_x^{\frac{\alpha}{2},\lambda}{\Phi_h},\,{}_{x}\mathbb{D}_b^{\frac{\alpha}{2},\lambda}{\Phi_h}\right)$. Using the fact  $\phi_{h,k}\in H^1_0(\Omega)$, one has
\begin{eqnarray} \label{toeplitzargument}
&&\left({}_{a}\mathbb{D}_x^{\frac{\alpha}{2},\lambda}{\Phi_h},\,{}_{x}\mathbb{D}_b^{\frac{\alpha}{2},\lambda}{\Phi_h}\right)
=-\left({}_a\mathbb{D}_x^{-(2-\alpha), \lambda}\left(\frac{d}{dx}+\lambda\right){\Phi_h},\,\left(\frac{d}{dx}-\lambda\right){\Phi_h}\right).
\end{eqnarray}
Obviously, all the elements of the vector function $\big(\frac{d}{dx}+\lambda\big){\Phi_h}$ or $\big(\frac{d}{dx}-\lambda\big){\Phi_h}$ are still  the dilation and translation of a single function. Letting $g(x)=e^{-\lambda x}x^{1-\alpha}$ in  Theorem \ref{Theoremnumercialeq}, matrix ${\bf A}_l$ is Toeplitz, and noting the compact support of $\phi_{h,k}$, only $N$ elements need to be calculated and stored.

It is also possible to calculate ${\bf A}_l$ of the high-degree element bases, with the computation and storage cost   $\mathcal{O}(N)$, instead of $\mathcal{O}(N^2)$. Based on  Theorem \ref{Theoremnumercialeq}, we try to take the bases as the dilations and translations of several known functions. For example, if one defines the compactly supported functions
\begin{eqnarray*}
&& H_1(x):=\left\{\begin{array}{lr} 2x^2-x &0\le x< 1,\\2x^2-7x+6&\quad 1\le x\le 2,\end{array}\right.\\[3pt]
&& H_2(x):=-4x^2+4x\quad 0\le x\le 1,
\end{eqnarray*}
and takes
\begin{eqnarray*}
&&\Phi_{h,1}=\Big\{\phi_{h,1,k}, k=0,\cdots,N-2\,\big|\phi_{h,1,k}=\frac{1}{\sqrt{h}}H_1\left(\frac{x-a}{h}-k\right)\Big\},\\
&&\Phi_{h,2}=\Big\{\phi_{h,2,k}, k=0,\cdots,N-1\,\big|\phi_{h,2,k}=\frac{1}{\sqrt{h}}H_2\left(\frac{x-a}{h}-k\right)\Big\},
\end{eqnarray*}
then $\Phi_h=\big\{\Phi_{h,1},\Phi_{h,2}\big\}$ is the bases of the quadratic element space, and it produces a block Toeplitz matrix. To compute ${\bf A}_l$, we can separate it into four parts, and in total only  $4\cdot2^J$ entries need to be computed and stored.

Though we have reduced the computation cost to $\mathcal{O}(N)$ to produce the corresponding stiffness matrix,  unlike the standard Riemann-Liouville operators, here even for the linear element approximation, the numerical integration must be used (for $\lambda>0$) for calculating ${\bf A}_l$. If only the one-sided derivative appears, the Petrov-Galerkin approximation will bring great convenience in generating the entries of the stiff matrix. Indeed, taking
\begin{eqnarray}\label{baseeq2}
\phi(x)=x_{+}-2(x-1)_++(x-2)_+,
\end{eqnarray}
for ${d}/{c}\ge a $ with $c>0,\,k\in \mathbb{N}^+$, one has
\begin{eqnarray*}
{}_a\mathbb{D}_x^{\pm\left(\alpha-1\right),\lambda}\left[e^{-\lambda x}(cx-d)_+^{k}\right]=c^{\pm\left(\alpha-1\right)}\frac{\Gamma(k+1)}{\Gamma(k\mp\left(\alpha-1\right)+1)}e^{-\lambda x}(cx-d)^{k\mp \left(\alpha-1\right)},
\end{eqnarray*}
where $x_+^k=\left(\max\{0,x\}\right)^k$. Therefore, using  (\ref{baseeq1}) and  (\ref{baseeq2}), it is easy to find that

\begin{eqnarray}\label{computeeq}
&&{}_a\mathbb{D}_x^{\pm\left(\alpha-1\right),\lambda}\left[e^{-\lambda x}\phi_{h,k}(x)\right]={}_a\mathbb{D}_x^{\pm\left(\alpha-1\right),\lambda}\left[e^{-\lambda x}\phi(\frac{x-a}{h}-k)\right]\\
&&~~~~~~~~~~~~~~~~~~~~~~~~~~~~~~~~~~~~={e^{-\lambda x}h^{\mp \left(\alpha-1\right)-\frac{1}{2}}}J\left(\frac{x-a}{h}-k\right),\nonumber
\end{eqnarray}
where $J(x)=\frac{1}{\Gamma(2\mp\left(\alpha-1\right))}\left(x^{1\mp\left(\alpha-1\right)}_{+}
-2\left(x-1\right)^{1\mp\left(\alpha-1\right)}_++\left(x-2\right)^{1\mp\left(\alpha-1\right)}_+\right)$.

Similar to the argument of (\ref{toeplitzargument}),  ${\bf A}_l=-\left({}_a{\mathbb D}_x^{-(2-\alpha),0}\frac{d}{dx}(\Phi_h),\,\frac{d}{dx}\left(\Phi_h\right)\right)$ is Toeplitz, and it can be exactly calculated by (\ref{computeeq}).
For the right tempered case, using the definitions (\ref{definition:1})--(\ref{definition:2}) and the symmetry  of $\phi(x)=\phi(2-x)$, it is easy to check that
\begin{eqnarray}
&&{}_x\mathbb{D}_b^{\pm\left(\alpha-1\right),\lambda}v(x)={}_a\mathbb{D}_y^{\pm\left(\alpha-1\right),\lambda}v(a+b-y)\Big|_{y=a+b-x},\\
&&\phi_{h,k}(a+b-x)=\phi_{h,N-2-k}(x).
\end{eqnarray}
Then it yields
\begin{eqnarray}
{}_x\mathbb{D}_b^{\pm\left(\alpha-1\right),\lambda}\left[e^{\lambda x}\phi_{h,k}(x)\right]=={e^{\lambda x}h^{\mp \left(\alpha-1\right)-\frac{1}{2}}}J\left(\frac{x-a}{h}-(N-2-k)\right).
\end{eqnarray}
And the similar results can be got for the right tempered matrix.

These techniques also apply to  the high-degree elements, such as, for the quadratic element, one  has
\begin{eqnarray*}
H_1(x)&=&2x_{+}^2-x_{+}-6(x-1)_{+}-2(x-2)^2_+-(x-2)_{+},\\
H_2(x)&=&-4x^2_{+}+4x_{+}+4(x-1)^2_{+}+4(x-1)_{+}.
\end{eqnarray*}
The Toeplitz or block Toeplitz structure also allows one to compute the matrix vector product with the cost $\mathcal{O}(N\log N)$ \cite{Chan:07}. Then the iterative method works efficiently.

\subsection{Time tempered fractional model}
In this subsection, we take the following time tempered fractional equation, defined on $\Omega\times (0,T]$, as a model:
\begin{eqnarray}\label{time_tempered_eq}
\frac{\partial }{\partial t}u(x,t)=K_{\gamma}\, {}_0\mathbb{D}_t^{1-\gamma,\lambda}\frac{\partial ^2}{\partial x^2}u(x,t)-\lambda u(x,t)
\end{eqnarray}
with  the initial condition $u(0,x)=g(x)$ and subject to the homogeneous Dirichlet boundary conditions $u(a,t)=0=u(b,t)$, where $ K_\gamma>0$ and $ 0<\gamma <1$.
It  can be regarded as the special case of the backward fractional Feynman-Kac equation \cite[eq. 20]{Carmi:10} with $pU=\lambda$. Using the techniques developed in \cite[Appendix]{Deng:15}, Eq. (\ref{time_tempered_eq}) can be rewritten as
\begin{eqnarray}\label{model2eq}
{}_0^C\mathbb{D}_t^{\gamma,\lambda}u(x,t)=K_\gamma \frac{\partial ^2}{\partial x^2}u(x,t).
\end{eqnarray}

\subsubsection{Model analysis}

By (\ref{equationee1}), one has  ${}_a^C\mathbb{D}_t^{\gamma,\lambda}u(x,t)={}_a\mathbb{D}_t^{\gamma,\lambda}u(x,t)-\frac{e^{-\lambda t}t^{-\gamma}}{\Gamma(1-\gamma)} g$; then combining Theorem \ref{theorem2} ($0<\mu=\frac{\gamma}{2}<\frac{1}{2}$) and Corollary \ref{corollaryeq35}, the variational method similar to \cite{Li:09, Li:10} can be developed to solve this equation.  But  the cost is high for the tempered case, and the regularity of $u(x,t)$ w.r.t. $ t$ is low. So, here we use a line method (given in (\ref{numericequ})) instead.  The finite element spatial discretization of (\ref{model2eq}) can be given as: find $u_h: [0,T]\to S_h, S_h\subset H_0^1(\Omega)$ such that
\begin{eqnarray} \label{semidiscrete}
\left({}_0^C\mathbb{D}_t^{\gamma,\lambda}u_h, v\right)=-K_\gamma \left(u^{\prime}_h, v^{\prime}\right) \quad \forall v\in S_h \times C([0,T]).
\end{eqnarray}
with $u_h(0)=g_h={\bf P}_h g\in S_h$, where $ {\bf P}_h g$ denotes the $L^2$ projection.

\begin{theorem}
 For $0<\gamma<1$, the space semi-discrete scheme (\ref{semidiscrete}) is unconditionally stable, and there
are
\begin{eqnarray}
&&\lambda\int_0^t\left|u_h\right|_{H^1(\Omega)}^2ds
+\frac{1}{2}\left|u_h(t)\right|_{H^1(\Omega)}^2 \le  \frac{1}{2}\left|g_h\right|^2_{H^1(\Omega)},\label{evolutioneq4}\\
&&~~\left\|u(t)\right\|_{L^2(\Omega)}\stackrel{<}{\sim}\left\|g_h\right\|_{L^2(\Omega)}+t^{\gamma}\left|g_h\right|^2_{H^1(\Omega)}.\label{evolutioneq5}
\end{eqnarray}
\end{theorem}
\begin{proof}
Choose $v={\rm D}_t^1u_h=\left(\frac{\partial}{\partial t}+\lambda\right)u_h$ in (\ref {semidiscrete}), i.e.,
\begin{eqnarray}
\left({}_0^C\mathbb{D}_t^{\gamma,\lambda}u_h,\,{\rm D}_t^1u_h\right)=-K_\gamma \left(u_h^{\prime},\,\left({\rm D}_t^1u_h\right)^{\prime}\right).
\end{eqnarray}
Notice that
\begin{eqnarray}\label{evolutioneq1}
~~\int_0^T\left(u_h^{\prime},\,\left({\rm D}_t^1u_h\right)^{\prime}\right)dt=\lambda\int_0^T\left|u_h\right|_{H^1(\Omega)}^2dt
+\frac{1}{2}\left(\left|u_h(T)\right|_{H^1(\Omega)}^2-\left|g_h\right|^2_{H^1(\Omega)}\right).
\end{eqnarray}
By Definition \ref{definition:3} and Theorem \ref{Theoremeq23}, one has
\begin{eqnarray}\label{evolutioneq2}
&&\int_0^T\left({}_0^C\mathbb{D}_t^{\gamma,\lambda}u_h,\,{\rm D}_t^1u_h\right)dt=\int_0^T\left({}_0\mathbb{D}_t^{-(1-\gamma),\lambda}\,{\rm D_t^1}u_h,\,{\rm D}_t^1u_h\right)dt\\
&&\ge \int_0^T\frac{\sin(\frac{\pi\gamma}{2})}{2\pi}\left\|{}_0\mathbb{D}_t^{-\frac{1-\gamma}{2},\lambda}\,{\rm D_t^1}u_h\right\|^2_{L^2(\Omega)}dt\ge 0.\nonumber
\end{eqnarray}
Combining  Definition \ref{definition:1}, (\ref{semigroup}), and Holder's inequality leads to
\begin{eqnarray}\label{evolutioneq31}
&&\left\|u_h(T)-e^{-\lambda t}g_h\right\|^2_{L^2(\Omega)}=\left\|{}_0\mathbb{D}_T^{-\frac{1+\gamma}{2},\lambda}{}_0\mathbb{D}_T^{-\frac{1-\gamma}{2},\lambda}\,{\rm D_t^1}u_h\right\|^2_{L^2(\Omega)}\\
&&\le \int_0^T\frac{(T-t)^{\gamma-1}}{\Gamma^2(\frac{1+\gamma}{2})}dt\int_0^T\left\|{}_0\mathbb{D}_t^{-\frac{1-\gamma}{2},\lambda}\,{\rm D_t^1}u_h\right\|^2_{L^2(\Omega)}dt\nonumber\\
&&\stackrel{<}{\sim} T^{\gamma}\int_0^T\left\|{}_0\mathbb{D}_t^{-\frac{1-\gamma}{2},\lambda}\,{\rm D_t^1}u_h\right\|^2_{L^2(\Omega)}dt.\nonumber
\end{eqnarray}
Now, replacing $T$ by $t$, then combining (\ref{evolutioneq1}) with (\ref{evolutioneq2}) results in (\ref{evolutioneq4}); and using (\ref{evolutioneq1})--(\ref{evolutioneq31}) and the triangle inequality leads to (\ref{evolutioneq5}).
\end{proof}

\subsubsection{Numerical implementation}
Let  $\left({\bf L}_h\varphi, \chi\right)=\left( \varphi^{\prime},\chi\prime\right)\,\, \forall \varphi,\chi\in S_h \times C([0,T])$. Then from (\ref{semidiscrete}), there exists
${}_0^C\mathbb{D}_t^{\gamma,\lambda}u_h=-K_\gamma {\bf L}_hu_h$,  by Laplace transform, whose formal solution can be given as
\begin{eqnarray}\label{numericequ}
u_h(x,t)=e^{-\lambda t}E_{\gamma,1}\left(-K_{\gamma}t^{\gamma}{\bf L}_h\right){\bf P}_h g,
\end{eqnarray}
where $E_{\gamma,\beta}(z)$ is the  Mittag-Leffler (ML) function \cite{Podlubny:99}, and the transform formulas $\mathscr{L}[{}_0^C\mathbb{D}_t^{\gamma,\lambda}u_h](s)=(s+\lambda)^{\gamma}\mathscr{L}[u_h](s)-(s+\lambda)^{\gamma-1}u_h(0)$ and
$\mathscr{L}[t^{\beta-1}E_{\gamma,\beta}(\pm p t^{\gamma})](s)=\frac{s^{\gamma-\beta}}{s^{\gamma}\mp p} \quad \Re(s)>\left|p\right|^{\frac{1}{\gamma}}$ are used. Exponential integrator and rational approximation have been well developed to solve the classical PDEs. In fact, because of the large storage requirement and computation cost of fractional operators, developing these type of algorithms for fractional PDEs makes more sense.

Denote
\begin{eqnarray}
{\rm \Pi}^{\gamma,\beta}v=t^{\beta-1}E_{\gamma,\beta}(-K_{\gamma}t^{\gamma}{\bf L}_h)v.
\end{eqnarray}
Define a branch cut along the negative axis and note that ${\bf L}_h$ is real symmetric positive define.  By the inverse Laplace transform,  for any $v \in L^2(\Omega)$, it follows that
\begin{eqnarray}\label{eqnarrayeeq1}
&&{\rm \Pi}^{\gamma,\beta}v=\frac{1}{2\pi i}\int_{\mathcal{C}}e^{st}{s^{\gamma-\beta}}\left({s^{\gamma}+K_{\gamma}{ {\bf L}_h }}\right)^{-1}v\,ds\nonumber\\
&&~~~~~~~~=\frac{t^{\beta -1}}{2\pi i}\int_{\mathcal{C}}e^{z}{z^{\gamma-\beta}}\left({z^{\gamma}+K_{\gamma}t^{\gamma}{\bf L}_h }\right)^{-1}v\,dz, \nonumber
\end{eqnarray}
where the path $\mathcal{C}$ is a deformed
Bromwich contour enclosing the negative axis in the anticlockwise sense \cite{Schmelzer:07,Trefethen:06}.  Since the resolvent estimate $\left\|z+{\bf L}_h\right\|_{L^2(\Omega)}\le {\Theta}\left|z\right|^{-1}$ \cite[Chapter 6]{Thomee:06},  ${z^{\gamma-\beta}}\left({z^{\gamma}+K_{\gamma}{\bf L}_h}\right)^{-1}$ is analytic in $\mathbb{C}/(-\infty, 0]$ for $0<\gamma<1$ and $\beta\ge \gamma$, and it tends to zero uniformly as $|z|\to \infty$.
 Then  one can replace $e^{z}$ with the $(N_1-1,N_1)$ type   Carathe{\'e}odory-Fej{\'e}r (CF) rational approximation $r_{N_1,N_1-1}=\sum_{k=1}^{N_1} \frac{c_k}{z-z_k}$ \cite{Schmelzer:07,Trefethen:06} to obtain an approximation of the form
\begin{eqnarray}\label{eqnarrr446}
{\rm \Pi}^{\gamma,\beta}v\approx-t^{\beta-1}\sum_{k=1}^{N_1}c_kz_k^{\gamma-\beta}\left({z_k^{\gamma}+K_{\gamma}t^{\gamma}{\bf L}_h }\right)^{-1}v,
\end{eqnarray}
where the poles $\left\{z_k\right\}$ and residues $\left\{c_k\right\}$ can be computed only once  and stored.
 To implement (\ref{eqnarrr446}),  one can solve $N_1$ elliptic problems: find $\hat{v}(z_k)=\left({z_k^{\gamma}+K_{\gamma}t^{\gamma}{\bf L}_h }\right)^{-1}v\in S_h$, such that
 \begin{eqnarray}\label{eqnarrayeeee}
 \Big(\left({z_k^{\gamma}+K_{\gamma}t^{\gamma}{\bf L}_h }\right)\hat{v}(z_k),\chi\Big)=\left(v,\chi\right) \quad \forall \chi\in S_h
 \end{eqnarray}
 in parallel, but a clever way is to cut  them to half by using the  complex conjugate nature of $(c_k,z_{k})$; and the standard preconditioning or multi-grid techniques can be applied to speed up the process. Finally, for any $ \mathcal{C}\subset \mathbb{C}/(-\infty, 0]$, one has
\begin{eqnarray}\label{boundeq}
\left\|{z^{\gamma-\beta}}\left({z^{\gamma}+K_{\gamma}t^{\gamma}{\bf L}_h}\right)^{-1}v\right\|_{L^2(\Omega)}\le \Theta \left|z\right|^{-\beta}\left\| v\right\|_{L^2(\Omega)}.
\end{eqnarray}
Then  by \cite[Theorem 5.2]{Schmelzer:07}, the theoretic convergence rate of the CF approximation  is  geometric. Our numerical experiments show that it always gives excellent results for $\beta \le 4$ while $N$ equals $14$ to $16$.

More generally, if the source term $f(x,t)\not=0$  in (\ref{model2eq}), the semi-discrete equation will become
\begin{eqnarray}
{}_0^C\mathbb{D}_t^{\gamma,\lambda}u_h=-K_\gamma {\bf L}_hu_h+{\bf P}_h f ,
\end{eqnarray}
and  $u_h(t)$  will be given as
\begin{eqnarray}\label{solutioneq}
u_h(x,t)=e^{-\lambda t}E_{\gamma,1}(-K_{\gamma}t^{\gamma} {\bf L}_h){\bf P}_h g+{\rm \Pi}^{\gamma}{\bf P }_h f,
\end{eqnarray}
where
\begin{eqnarray}\label{eqnarraeq}
{\rm \Pi}^{\gamma}{\bf P }_h f=\int_0^t(t-s)^{\gamma-1}E_{\gamma,\gamma}(-K_{\gamma}(t-s)^{\gamma}{\bf L}_h)e^{\lambda(s-t)}{\bf P }_h fds.
\end{eqnarray}
By the equation (1.100) in \cite[ p. 25]{Podlubny:99}, for $\nu>0$ and $\beta>0$,  it follows that
\begin{eqnarray} \label{eqnarra3}
&&\int_a^t (t-s)^{\beta-1}E_{\gamma,\beta}(-q (t-s)^{\gamma})(t-a)^{\nu -1}ds\\
&&={\Gamma(\nu)}(t-a)^{\beta+\nu-1}E_{\gamma,\beta+\nu}(-q (t-a)^{\gamma}).\nonumber
\end{eqnarray}
Therefore, if  $f(x,t)$ has the form $e^{-\lambda t}\left(t^{\nu_1-1}g_1(x)+t^{\nu_2-1}g_2(x)+\cdots\right)$,  one can remove the integral symbol in (\ref{eqnarra3}) exactly. Otherwise if $f(x,t) $ is piecewise smooth w.r.t. $ t$, the   subdivision  low-order interpolation can be used, i.e., letting $0=t_0<\cdots<t_{M-1}<t_M=t$ be the partition of $\left[0,t\right]$, and
\begin{eqnarray}
h_k(x,s)=\sum_{l=0}^{m_k}c_{m_k,l}^L(s-t_k)^l=\sum_{l=0}^{m_k}c_{m_k,l}^R(s-t_{k+1})^l
\end{eqnarray}
be the $ m_k$ degree interpolation of $e^{\lambda(s-t)}{\bf P}_h f$ in interval $[t_k,t_{k+1}]$, then
\begin{eqnarray}
{\rm \Pi}^{\gamma}{\bf P }_h f
&\approx&\sum_{k=0}^{M-1}\left(\int_{t_k}^t-\int_{t_{k+1}}^t\right)(t-s)^{\gamma-1}E_{\gamma,\gamma}(-K_{\gamma}(t-s)^{\gamma}{\bf L}_h)e^{\lambda(s-t)}h_k(s)ds\nonumber\\
&=&\sum_{k=0}^{M-1}\sum_{l=0}^{\max\{m_k,m_{k-1}\}}\Gamma(l+1)(t-t_k)^{\gamma+l}E_{\gamma,\gamma+l+1}\left(-K_\gamma(t-t_k)^\gamma{\bf L}_h\right)\left(c_{m_k,l}^L-c_{m_{k-1},l}^R\right)\nonumber\\
&&+\sum_{l=0}^{m_0}\Gamma(l+1)t^{\gamma+l}E_{\gamma,\gamma+l+1}\left(-K_\gamma t^\gamma {\bf L}_h\right)c_{m_0,l}^L,\nonumber
\end{eqnarray}
where we take $ c_{m_k,l}^L=0 $ for $l>m_k$ and $ c_{m_{k-1},l}^R=0 $ for $l>m_{k-1}$, respectively. Then every
$(t-t_k)^{\gamma+l}E_{\gamma,\gamma+l+1}\left(-K_\gamma(t-t_k)^\gamma{\bf L}_h\right)v$ can be handled with the same $(c_k,z_k)$
 by (\ref{eqnarrr446}), and the  final error will be dominated by the  interpolation of  $e^{\lambda(s-t)}{\bf P}_h f$ (which can be controlled in advance, even to obtain the partition  and interpolation points adaptively), superior to the finite difference method or the predictor-corrector method, which  usually depends on the regularity of the exact solution $u_h(t)$ or ${}_0^C\mathbb{D}_t^{\gamma,\lambda}u_h$ \cite{Deng:15}.

The  high-order element or  the composite spectral interpolation (even the best uniform approximation) can be used to get a fast and better approximation  if ${\bf P}_h f$ has a good regularity. For example, one can choose $\xi_{M,j}=\cos\frac{(2j-1)\pi}{2M},j=0,1,\cdots,M$, the Chebshev Gauss nodes on $[-1,1]$, and let
$\left\{s_{M,j}; s_{M,j}=\frac{t}{2}\left(1+\xi_{M,j}\right)\right\}$ be the  interpolation points, then the Chebshev Lagrange interpolation of $e^{\lambda(s-t)}{\bf P}_h f$ on the whole interval $[0,t]$ can be given as
\begin{eqnarray}   \label{eqnareee}                                                                                                                                                                                                                                            ~~~h(x,s)=\sum_{j=0}^{M} e^{\lambda(s_{M,j}-t)}{\bf P}_h f(x,s_{M,j})L_{M,j}(s)=\sum_{j=0}^{M} c_{M,j}\, s^j.
\end{eqnarray}
 Using (\ref{solutioneq}) and (\ref{eqnarra3}), $u_h(x,t)$ can be approximated by
\begin{eqnarray}\label{eqeeter4}
&&{u_h}(x,t)\approx e^{-\lambda t}E_{\gamma,1}(-K_{\gamma}t^{\gamma} {\bf L}_h){\bf P}_h g+{\rm \Pi}_{high}^{\gamma}{\bf P}_h f,
\end{eqnarray}
where
\begin{eqnarray}
{\rm \Pi}_{high}^{\gamma}{\bf P}_h f=\sum_{j=0}^{M}\Gamma(j+1)t^{\gamma+j}E_{\gamma,\gamma+j+1}(-K_{\gamma}t^{\gamma}{\bf L}_h)c_{M,j}.
\end{eqnarray}
But to handle ${\rm \Pi}^{\gamma,\gamma+j+1}v$ for big $\beta=\gamma+j+1$ and small $N_1$, the CF approximation  is distorted  due to  the fact that  ${z^{\gamma-\beta}}\left({z^{\gamma}+K_{\gamma}{\bf L}_h}\right)^{-1}$ decays so fast that the left-most nodes make a negligible contribution (the same happens when approximating the gamma function \cite[Fig. 4.3]{Schmelzer:07}). One can  overcome this by fine-tuning the integral in a manner specific to $\gamma,\beta$. For  $\sigma>0$,  let $\mathcal{C}$ be the simplest  parabolic contour (PC) $z(p)=\sigma(ip+1)^2, p\in\mathbb{R}$ given in \cite{Garrappa:15}. The fast decay of $\left|e^{z(p)}\right|$ for $\left|p\right|\to \infty$ allows one to produce an approximation of ${\rm \Pi}^{\gamma,\beta}v$ as
\begin{eqnarray}
&&{\rm \Pi}^{\gamma,\beta}v=\frac{t^{\beta -1}}{2\pi i}\int_{-\infty}^{\infty}z^{\prime}(p)e^{z(p)}{z^{\gamma-\beta}}(p)\left({z^{\gamma}(p)+K_{\gamma}t^{\gamma}{\bf L}_h }\right)^{-1}v\,dp\nonumber\\
&&\approx\frac{t^{\beta -1}\tau_1}{2\pi i}\sum_{k=-N_1}^{N_1}z_k^{\prime}e^{z_k}{z_k^{\gamma-\beta}}\left({z_k^{\gamma}+K_{\gamma}t^{\gamma}{\bf L}_h }\right)^{-1}v\nonumber\\
&&=\frac{t^{\beta -1}\tau_1}{2\pi i}\sum_{k=0}^{N_1}\nu_kz_k^{\prime}e^{z_k}{z_k^{\gamma-\beta}}\left({z_k^{\gamma}+K_{\gamma}t^{\gamma}{\bf L}_h }\right)^{-1}v\stackrel{\triangle}{=}{\rm \Pi}_{N_1}^{\gamma,\beta}v,\nonumber
\end{eqnarray}
where $z_k=z(p_k), z^{\prime}_k=z^{\prime}(p_k)$, $p_k=k\cdot\tau_1$, and $\nu_k=\left\{\begin{array}{ll}1&k=0\\ 2&k\ge 1\end{array}\right.$.
\begin{lemma}
For any  given $\sigma>0$,  the PC discretaization $ {\rm \Pi}_{N_1}^{\gamma,\beta}v$ is stable for $\beta\ge \frac{1}{2}$ and $\tau_1\le1$, i.e.,
\begin{equation}
\left\|{\rm \Pi}_{N_1}^{\gamma,\beta}v\right\|_{L^2(\Omega)}\stackrel{<}{\sim}\frac{t^{\beta -1}}{2\pi }e^{\sigma}\sigma^{-\beta+1}\left\|v\right\|_{L^2(\Omega)}.
\end{equation}
\end{lemma}
\begin{proof} By (\ref{boundeq}), it follows that
\begin{eqnarray}
 &&\left\|{\rm \Pi}_{N_1}^{\gamma,\beta}v\right\|_{L^2(\Omega)}\le \Theta\frac{t^{\beta -1}\tau_1}{2\pi }\sum_{k=-N_1}^{N_1}\left|z_k^{\prime}e^{z_k}\right| \left|z_k\right|^{-\beta}\left\| v\right\|_{L^2(\Omega)}\\
 &&=\Theta\frac{t^{\beta -1}\tau_1}{2\pi }e^{\sigma}\sigma^{-\beta+1}\left(1+2\sum_{k=1}^{N_1}e^{-\sigma p_k^2}\left(1+p_k^2\right)^{-\beta+\frac{1}{2}}\right)\left\|v\right\|_{L^2(\Omega)}\nonumber\\
 &&\le \Theta\frac{t^{\beta -1}}{2\pi }e^{\sigma}\sigma^{-\beta+1}\left(\tau_1+2\int_0^{\infty}e^{-\sigma p^2}\left(1+p^2\right)^{-\beta+\frac{1}{2}}dp\right)\left\|v\right\|_{L^2(\Omega)},\nonumber
\end{eqnarray}
which completes the proof.
\end{proof}

 We  extend the function $z(p)$ defined for $p\in\mathbb{R}$ as analytic function in a strip $Y=\left\{p=\zeta+i\eta, \,a_{-}<\eta<a_{+}\right\}$. It is easy to check that  the neighbourhood $Z=\left\{z(p),\, p\in Y\right\}$ of  the contour $\mathcal{C}$ lies in $\mathbb{C}/(-\infty,0]$ for any $0<a_{+}<1$ and $a_{-}<0$. Then according to  the convergence  result of the trapezoidal rule for the integral over the real line (see the proof of  Theorem 5.1 in \cite[Section 5]{Trefethen:14}), for $\beta\ge\frac{1}{2}$, it holds that
 \begin{eqnarray}
 \left\|{\rm \Pi}^{\gamma,\beta}v-{\rm \Pi}_{N_1}^{\gamma,\beta}v\right\|_{L^2(\Omega)}\le\frac{M_{-}}{e^{-2\pi a_{-}/\tau_1}-1}+\frac{M_{+}}{e^{2\pi a_{+}/\tau_1}-1}+TE,
 \end{eqnarray}
 where
 \begin{eqnarray*}
 &&M_{\pm}=\Theta\frac{t^{\beta -1}}{\pi}\int_{-\infty}^{\infty}\left|z^{\prime}\left(\zeta+ia_{\pm}\right)e^{z\left(\zeta+ia_{\pm}\right)}\right| \left|z\left(\zeta+ia_{\pm}\right)\right|^{-\beta}\left\|v\right\|_{L^2(\Omega)}d\zeta\\
  &&\le\Theta\frac{2t^{\beta-1}\sigma^{-\beta+1}}{\pi}e^{\sigma
 (1-a_{\pm})^2}\left(\left((1-a_{\pm})^2\right)^{-\beta+\frac{1}{2}}\int_{0}^1e^{-\sigma\zeta^2}d\zeta
 +\int_{1}^{\infty}e^{-\sigma\zeta^2}d\zeta\right)\left\|v\right\|_{L^2(\Omega)}\nonumber\\
 &&\le\Theta\frac{2t^{\beta-1}\sigma^{-\beta+1}}{\pi}e^{\sigma(1-a_{\pm})^2}\left((1-a_{\pm})^{-2\beta+1}
 +\frac{\sqrt{\pi}}{2\sigma^{1/2}}\right)\left\|v\right\|_{L^2(\Omega)},\nonumber
 \end{eqnarray*}
 and
 \begin{eqnarray*}
 &&TE= \Theta\frac{t^{\beta -1}\tau_1}{\pi }\sum_{k\ge N_1+1}\left|z_k^{\prime}e^{z_k}\right| \left|z_k\right|^{-\beta}\left\| v\right\|_{L^2(\Omega)}\\
  &&\le\Theta\frac{t^{\beta -1}}{\pi }e^{\sigma}\sigma^{-\beta+1}\int_{N_1\tau_1}^{\infty}e^{-\sigma p^2}\left(1+p^2\right)^{-\beta+\frac{1}{2}}dp\left\|v\right\|_{L^2(\Omega)}\\
  &&\le\Theta\frac{t^{\beta -1}}{2\sqrt{\pi} }e^{\sigma}\sigma^{-\beta+1/2}e^{-\sigma(N_1\tau_1)^2}\left\|v\right\|_{L^2(\Omega)}.
 \end{eqnarray*}

 For any $\sigma\tau_1<\pi$, by choosing $a_{-}=1-\frac{\pi}{\sigma \tau_1}<0$, it yields $e^{\sigma(1-a_{-})^2+\frac{2\pi a_{-}}{\tau_1}}\le e^{-\frac{\pi^2}{\sigma\tau_1^2}+\frac{2\pi}{\tau_1}}$.  One can  balance the orders of magnitude of three error terms to estimate the optimal truncation point, i.e.,
 \begin{eqnarray}
e^{-\frac{\pi^2}{\sigma\tau_1^2}+\frac{2\pi}{\tau_1}}\simeq e^{\sigma(1-a_{+})^2-\frac{2\pi a_{+}}{\tau_1}}(1-a_{+})^{-2\beta+1}\simeq e^{\sigma-\sigma(N_1\tau_1)^2}.
 \end{eqnarray}
 Noting that for $a_+\to 1$, it follows that
 \begin{eqnarray*}
e^{\sigma(1-a_{+})^2-\frac{2\pi a_{+}}{\tau_1}}(1-a_{+})^{-2\beta+1}\sim  e^{-\frac{2\pi a_{+}}{\tau_1}}(1-a_{+})^{-2\beta+3}\stackrel{<}{\sim}  e^{-\frac{2\pi a_{+}}{\tau_1}}(1-a_{+})^{2(\gamma-\beta+1)}.
 \end{eqnarray*}
So the algorithms developed in \cite[Subsection 3.2.2]{Garrappa:15} for computing the ML function can be directly used here or   after replacing the corresponding $\beta$ by $\beta+\gamma-\frac{1}{2}$; they produce good numerical results for all $\beta\ge\frac{1}{2}$.

 An alternative rational approximation could also be developed based on the Dunford-Taylor integral (DTI) representation, it holds that
\begin{eqnarray}
{\rm \Pi}^{\gamma,\beta}v=\frac{t^{\beta-1}}{2\pi i}\int_{\mathcal{C}}E_{\gamma,\beta}(-K_{\gamma}t^{\gamma}z)\left(z{\bf I}-{\bf L}_h\right)^{-1}v dz,
\end{eqnarray}
where $\mathcal{C}$ is a closed contour  enclosing the spectrum of ${\bf L}_h$, and $\hat{v}(z)=\left(z{\bf I}-{\bf L}_h\right)^{-1}v\in S_h$ can be computed by
\begin{eqnarray}
\Big(\left({z{\bf I}-{\bf L}_h }\right)\hat{v}(z),\chi\Big)=\left(v,\chi\right)\quad \forall \chi\in S_h.
\end{eqnarray}
Note that the ML function is entire. Theoretically, one can choose any sufficiently large circle $\mathcal{C}$ to gain a fast exponential convergence; see \cite[Theorem 18.1]{Trefethen:14}. In practice, it fails, due to the fast increase of $\left|E_{\gamma,\beta}(z)\right|$ for $Re(z)\to\infty$ with $ \left|\arg(z)\right|\le \nu$ and $ \nu\in\left(\frac{\pi\gamma}{2},\pi\gamma\right)$ \cite[Theorem 1.5]{Podlubny:99}. Therefore, a stable and wise way is based on \cite[Theorem 1.6]{Podlubny:99}
 \begin{eqnarray}
 \left|E_{\gamma,\beta}(z)\right|\le \frac{C}{1+\left|z\right|}, \quad \left|\arg(z)\right|\in \left[\nu,\pi\right]
 \end{eqnarray}
to select a circle $\mathcal{C}$ lying in the right half $z$-plane, not very closing to the original point (to reduce the rand error).  Thus the techniques presented in \cite[Method 1]{Hale:08} perfectly work here to produce  such a  discretization of ${\rm \Pi}^{\gamma,\beta}v $ with the number of quadrature nodes needed to obtain a specified accuracy increasing asymptotically as $\log(\sigma_{\max}/\sigma_{\min})$ \cite[Theorem 2.1]{Hale:08}, where $\sigma_{\max}$ and $\sigma_{\min}$ are the largest and smallest eigenvalue of the corresponding matrix  of ${\bf L}_h$, respectively.   Our numerical experiments show that it can  produce the same accuracy with the first two approaches, but usually a longer time is required (being the same as the observation in \cite{Burrage:12}).
\begin{remark}
 The quadrature points $\{z_k\}$ are independent of the  time, so they can be pre computed and stored only once in  the case  $ f\not= 0$. Moreover, the points $\{z_k\}$   don't depend on $\gamma$ and $\beta$ in the CF and DTI methods; but for the PC method,  they must be  pre computed for different $\gamma$ and $\beta$, at the same time the number of points is much less than the one of the DTI method. Finally, it can be noted that the methods developed in this section have no restriction in the space dimensions, and they can also be directly applied to the fractional case, such as, the Riesz derivative and the fractional Laplace operator \cite{Xu:14,Yang:11,Yang:10}, and  the general strongly elliptic operator with $0<\gamma\le \frac{1}{2}$ (for the corresponding resolvent estimate and spectral distribution, see, e.g., \cite[Chapter 6]{Thomee:06} and \cite{Zhang:15}).
\end{remark}

\section{Numerical results}
In this section, the numerical experiments are carried out to assess the computational performance and effectiveness of the numerical schemes.
In the following, we always choose $\Omega=(0,1)$ with   $N=2^J$ space partitions. All numerical experiments are run in
MATLAB 7.11 (R2010b) on a PC with Intel(R) Core (TM)i7-4510U 2.6 GHz processor
and 8.0 GB RAM. The codes to  produce the quadrature points $\{z_k\}$  in the CF, PC and DTI schemes are adapted from \cite{Trefethen:06,Garrappa:152,Hale:08}, respectively; we choose $N_1=16$ for the PC scheme and $N_1=10\cdot\lceil\log(\sigma_{\max}/\sigma_{\min})+3\rceil$ for the DTI scheme; and the parameters of the PC scheme is adaptively produced by code itself.

\begin{example} \label{example1}
Consider (\ref{model1})  with  $u(0)=u(1)=0, p=1, m(x)=q\lambda^{\alpha-1}(1-x)$ and $c(x)=0$.  The right hand term $f(x)$ is derived from the exact solution $u(x)=\left(1-x\right)^\beta-e^{\lambda x}(1-x)$.
\end{example}

We take the linear element space as $S_h$ and use the norm defined by
\begin{eqnarray*}
\left\|u\right\|_{{\alpha}/{2}}^1=\sqrt{-\left({}_{a}\mathbb{D}_x^{\alpha/2,0}{u},{}_{x}\mathbb{D}_b^{\alpha/2,0}{u}\right)}\sim \left\|u\right\|_{H^{\alpha/2}(\Omega)}.
\end{eqnarray*}
 The $\left\|\cdot\right\|_{{\alpha}/{2}}^1$ errors and convergence rates of the Galerkin and Petrov-Galerkin method are shown in Table \ref{tab:1-1}, which well confirm the theoretical prediction  (\ref{eqarrayeq}). And the corresponding $L^2$ ones are given in Table \ref{tab:1-2}.
\begin{table}\fontsize{8.5pt}{12pt}\selectfont
\begin{center}
 \caption{Numerical results {\rm ($\left\|\cdot\right\|_{{\alpha}/{2}}^1$-error)} of the Galerkin (G) and Petrov-Galerkin (P-G) for Example \ref{example1}  with $ q=0$ and $\beta=3$.}
\begin{tabular}{cc|cc|cc|cc|cc}
  \hline
  $type$  & $J$   &\multicolumn{2}{c|}{$\alpha=1.4,\lambda=3$ }  &\multicolumn{2}{c|}{$\alpha=1.4,\lambda=5$ }   & \multicolumn{2}{c|}{$\alpha=1.8,\lambda=3$ } & \multicolumn{2}{c}{$\alpha=1.8,\lambda=5$ }  \\
  \cline{3-10}
             &        & Err      & Rate                  &Err  & rate            &Err  & Rate            &Err  & Rate    \\
   \hline
                      &$6$    &2.0583e-02   & ---        & 1.9611e-01     &---        & 9.1200e-02     & ---      &  8.6035e-01  & ---     \\
  G                  &$7$    & 8.2721e-03   &  1.3151    & 7.8394e-02     & 1.3229    & 4.2439e-02     & 1.1036   & 3.9984e-01    &  1.1055\\
                     &$8$     &3.3402e-03   &  1.3083    & 3.1543e-02     & 1.3134    &1.9771e-02      &1.1020     &  1.8612e-01     &1.1032 \\
\hline
                     &$6$    & 7.6343e-03   & ---         & 1.3508e-02    & ---       &  3.4052e-02     & ---     &  6.0270e-02    & ---       \\
P-G                  &$7$    &3.0794e-03    &  1.3098     & 5.4459e-03    & 1.3106    & 1.5859e-02       & 1.1024   &  2.8046e-02      & 1.1036   \\
                     &$8$    &1.2494e-03    &  1.3014     & 2.2109e-03    & 1.3006    &  7.3922e-03   &1.1012     &  1.3074e-02    & 1.1011  \\
 \hline
\end{tabular} \label{tab:1-1}
\end{center}
\end{table}

\begin{table}\fontsize{8.5pt}{12pt}\selectfont
\begin{center}
 \caption{Numerical results {\rm ($L^2$-error)} of the Galerkin (G) and Petrov-Galerkin (P-G) for Example \ref{example1}  with $ q=2$.}
\begin{tabular}{cc|cc|cc|cc|cc}
  \hline
  $type$  & $J$   &\multicolumn{2}{c|}{$\alpha=1.4,\lambda=3,\beta=3$ }  &\multicolumn{2}{c|}{$\alpha=1.4,\lambda=5,\beta=3$ }   & \multicolumn{2}{c|}{$\alpha=1.8,\lambda=0,\beta=1.1$ } & \multicolumn{2}{c}{$\alpha=1.8,\lambda=5,\beta=1.1$ }  \\
  \cline{3-10}
             &        & Err      & Rate                  &Err  & rate            &Err  & Rate            &Err  & Rate    \\
   \hline
                      &$6$    &  4.0085e-04   & ---      & 3.9560e-03     &---        & 2.8614e-05     & ---      & 5.2145e-03      & ---     \\
  G                  &$7$    & 9.5772e-05   & 2.0654    & 9.2271e-04     & 2.1001    & 9.4361e-06     & 1.6004   & 1.2063e-03      &2.1120\\
                     &$8$     &2.3343e-05   &  2.0366    & 2.2145e-04     &  2.0589    &3.1123e-06     & 1.6002    & 2.8209e-04     &2.0964 \\
\hline
                     &$6$    &   5.8454e-04    & ---         &  1.0838e-03  & ---      & 2.8603e-05   & ---     &  4.2569e-04  & ---       \\
P-G                  &$7$    & 1.4791e-04   &   1.9825     &  2.7674e-04   &1.9695    & 9.4325e-06      & 1.6004   &1.1182e-04   &1.9286   \\
                     &$8$    & 3.7258e-05   &   1.9891     &7.0079e-05    & 1.9815    &   3.1111e-06    &1.6002     &   2.9333e-05   &  1.9306  \\
 \hline
\end{tabular} \label{tab:1-2}
\end{center}
\end{table}

A well conditional number and ¡®bunching of eigenvalues¡¯ of the matrix equation usually mean the good numerical stability and  the faster iteration convergence speed.
The continuity and  coerciveness of $A\left(\cdot,\cdot\right)$ mean the algebraic system
\begin{eqnarray}\label{eqnarreeee}
{\bf A}{\rm U}={\rm F}
\end{eqnarray}
corresponding to (\ref{weakform}) has the condition number $\mathcal{O}(2^{J\alpha})$ (i.e., $\mathcal{O}(N^{\alpha}$)). In fact,  for  any $u=\Phi_h {\rm U}, v=\Phi_h {\rm V} \in S_h$, it follows that
\begin{eqnarray*}
&&\left({\bf A}{\rm U},{\rm V}\right)_{l_2}=A\left(u,v\right)\stackrel{<}{\sim}\left\|u\right\|_{H^{\frac{\alpha}{2}}(\Omega)}\left\|v\right\|_{H^{\frac{\alpha}{2}}(\Omega)}\stackrel{<}{\sim}2^{J\alpha}\left\|u\right\|_{L^2(\Omega)}\left\|v\right\|_{L^2(\Omega)}\sim 2^{J\alpha}\left\|{\rm U}\right\|_{l_2}\left\|{\rm V}\right\|_{l_2},\nonumber\\
&&~~~\left({\bf A}{\rm U},{\rm U}\right)_{l_2}=A\left(u,u\right)
\stackrel{>}{\sim}\left\|u\right\|^2_{H^{\frac{\alpha}{2}}(\Omega)}\stackrel{>}{\sim}\left\|u\right\|_{L^2(\Omega)}^2\sim\left\|{\rm U}\right\|_{l_2}^2,
\end{eqnarray*}
where the inverse estimate and the fact that $\frac{1}{3}\left\|{\rm U}\right\|_{l_2}\le \left\|u\right\|_{L^2(\Omega)}\le \left\|{\rm U}\right\|_{l_2}$ are used. Therefore, $\left\|{\bf A}^{-1}\right\|_{l_2}\stackrel{<}{\sim}1\stackrel{<}{\sim}\left\|{\bf A}\right\|_{l_2}\stackrel{<}{\sim}2^{J\alpha}$. If  a new basis  $\Psi_h$  of $S_h$ can be chosen such that for any
 $u=\Psi_h {\rm V^{\star}}\in S_h$,  $\left\|v\right\|_{H^{\frac{\alpha}{2}}(\Omega)}\sim \left\|{\rm V^\star}\right\|_{l_2}$ with the constant independent of $J$ (i.e., Riesz basis, see \cite[p. 463]{Cohen:00}), then the corresponding algebraic matrix will have the well condition number.  To do this, we introduce the multiscale basis of $S_h$.
 Let
 \begin{eqnarray}\label{saceeqa2}
 \psi^*=\phi(2x),\quad \Psi^*_j=\{\psi^*_{j,k}; \psi^*_{j,k}=2^{\frac{j}{2}}\psi^*(2^jx-k)\}_{k=0}^{2^j-1}.
 \end{eqnarray}
  Then $\Psi^*_h=\bigcup_{j=0}^{J-1}\Psi^*_j$ form  the   multiscale wavelet basis (i.e., the Schauder hierarchical basis) of $S_h$, and it holds that $\Psi^*_h=\Phi_h{\bf W}$, where ${\bf W}$ denotes the fast wavelet transform  (FWT) matrix, which can be obtained by (\ref{sacleq1}) and (\ref{saceeqa2}) and  implemented by the FWT algorithm with the cost $\mathcal{O}(2^J)$ (see \cite[p. 433-437]{Cohen:00}).  Taking $\psi_{j,k}=\psi^*_{j,k}/\sqrt{A(\psi^*_{j,k},\psi^*_{j,k})}$ and $\Psi_j=\left\{\psi_{j,k}\right\}_{k=0}^{2^j-1}$, by \cite[Theorem 30.7 and  p. 605]{Cohen:00},
$\Psi_h=\bigcup_{j=0}^{J-1}\Psi_j$ is a Riesz basis of $S_h$, and the corresponding algebraic system under this basis can be given as the  preconditioned form of (\ref{eqnarreeee}), i.e.,
\begin{eqnarray}
\underbrace{{\bf D}{\bf W}^T{\bf A} {\bf W} {\bf D}}_{{\bf A}^\star}\underbrace{\left({\bf D}^{-1}{\bf W}^{-1} {\rm U}\right)}_{\rm U^\star}=\underbrace{{\bf D}{\bf W}^T{\rm F}}_{\rm F^\star},
\end{eqnarray}
where ${\bf A}^\star, {\rm U^\star} $ and ${\rm F^\star}$ denotes the stiffness matrix, unknown vector   and  the right term  under  $\Psi_h$, respectively, and
\begin{eqnarray*}
{\bf D}={\rm diag}\left(\underbrace{d_{0,0}}, \underbrace{d_{1,0},\cdots,d_{1,1}},\cdots,\underbrace{d_{J,0},\cdots,d_{J,2^J-1}}\right)\quad
d_{j,k}=1/\sqrt{A(\psi^*_{j,k},\psi^*_{j,k})}.
\end{eqnarray*}
The diagonal matrix ${\bf D}$   can be produced with  the cost  $\mathcal{O}(2^J)$. 
 In the process of preconditioning, ${\bf W}{\rm Q}$ and ${\bf W}^T {\rm Q}$ can be computed by the FWT with the cost  $\mathcal{O}(2^J)$ (see \cite[p.433-437]{Cohen:00} and \cite[Chapter 6]{Urban:09}); and ${\bf A} {\rm Q}$ can be implemented by the FFT with the cost $\mathcal{O}(J\cdot 2^J)$ \cite{Chan:07,Pang:12,Wang:10}, where ${\rm Q}$ denotes a vector.
The condition numbers of the algebraic system and the CPU time  of the GMRES iteration (before and after preconditioning)  are listed in Table \ref{tab:1-3}. The stopping
criterion for solving the linear systems is
$\frac{\left\|r(k)\right\|_{l_2}}{\left\|r(0)\right\|_{l_2}}\le 1e-8$, where $r(k)$ is the residual vector after $k$ iterations. We also display the spectral distribution  in Figure \ref{fig1}.

 \begin{table}\fontsize{8.5pt}{12pt}\selectfont
\begin{center}
 \caption{The condition number (Con-num) and the CPU time  of  GMRES  for Example \ref{example1} with $q=0,\,\alpha=1.7$ and $\lambda=3$.}
\begin{tabular} {c|cc|cc|cc|cc}  \hline
$J$  &\multicolumn{2}{c}{Before pre-} &\multicolumn{2}{c|}{ GMRES}   &\multicolumn{2}{c}{After Pre-}     &\multicolumn{2}{c}{GMRES }     \\\cline{2-9}
         &   Con-num &  Rate      & Iter           & CPU(s)          &Con-num    &Rate       & Iter & CPU(s)    \\ \hline

    $7$  &  2.2768e+03  &  ---   & 1.2700e+02    & 0.0960           &   1.6869   & ---        & 14.0     & 0.0113              \\
    $8$  &    7.4179e+03  &  1.7040 &  2.5500e+02   & 0.3897           &  1.7816   &  0.0788     & 14.0     & 0.0109       \\
    $9$  &     2.4135e+04  &  1.7020 & 5.1100e+02    &1.6416            &   1.8642   & 0.0654      & 15.0     & 0.0156           \\
    \hline
   \end{tabular}\label{tab:1-3}
\end{center}
\end{table}

\begin{figure}
\begin{center}
\includegraphics[width=1.2in,height=1.2in,angle=0]{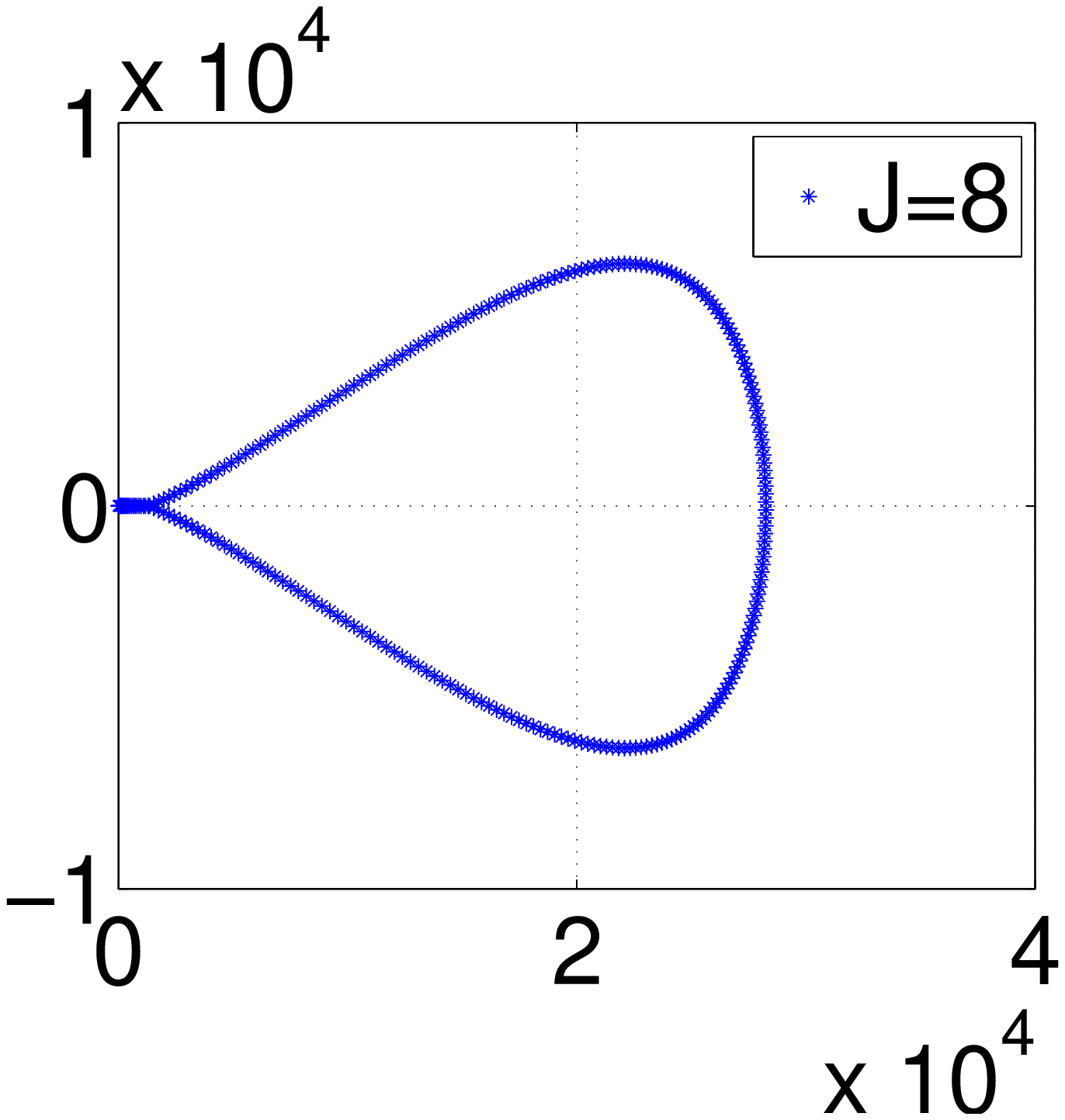}
\includegraphics[width=1.2in,height=1.2in,angle=0]{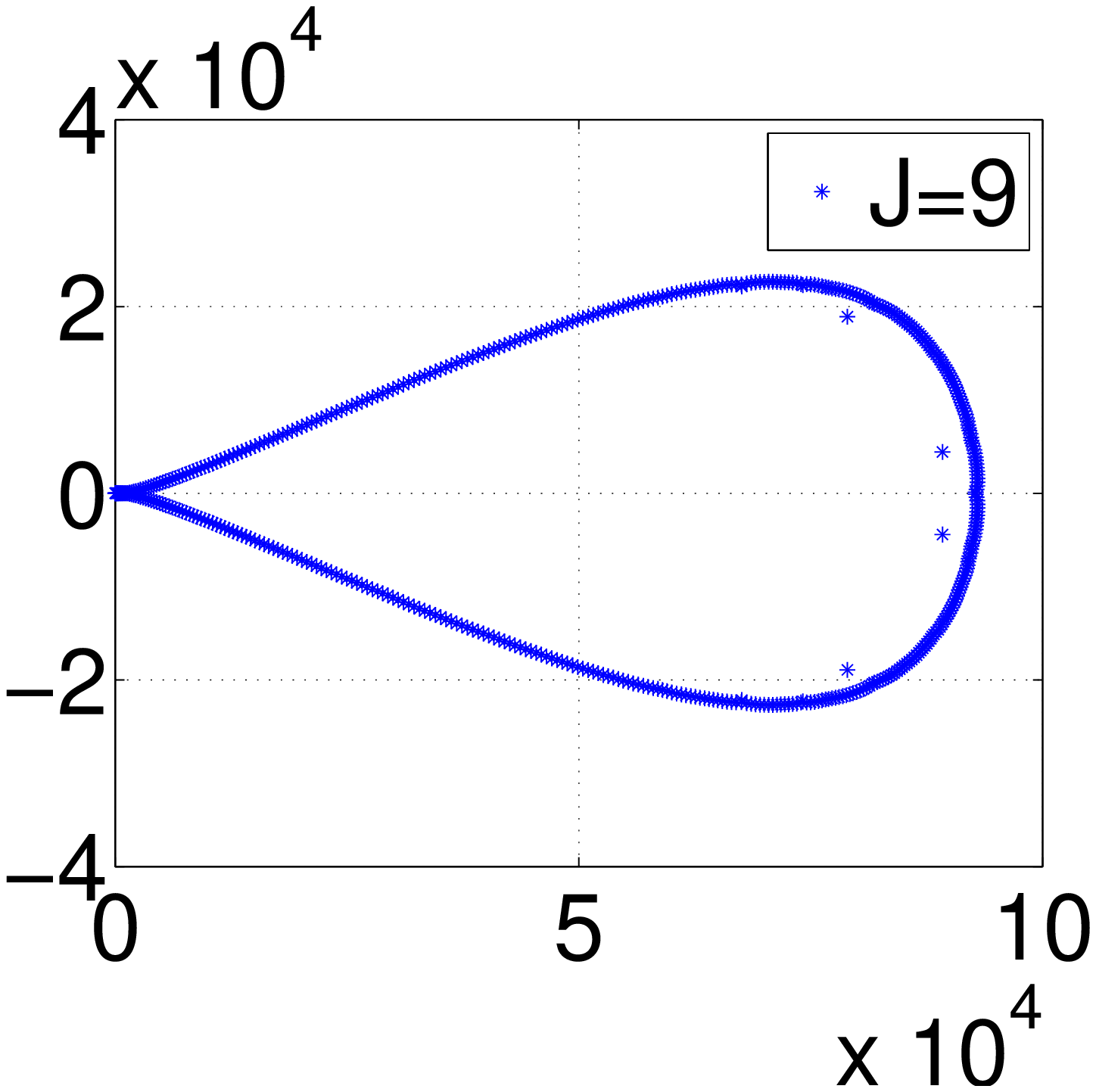}
\includegraphics[width=1.2in,height=1.2in,angle=0]{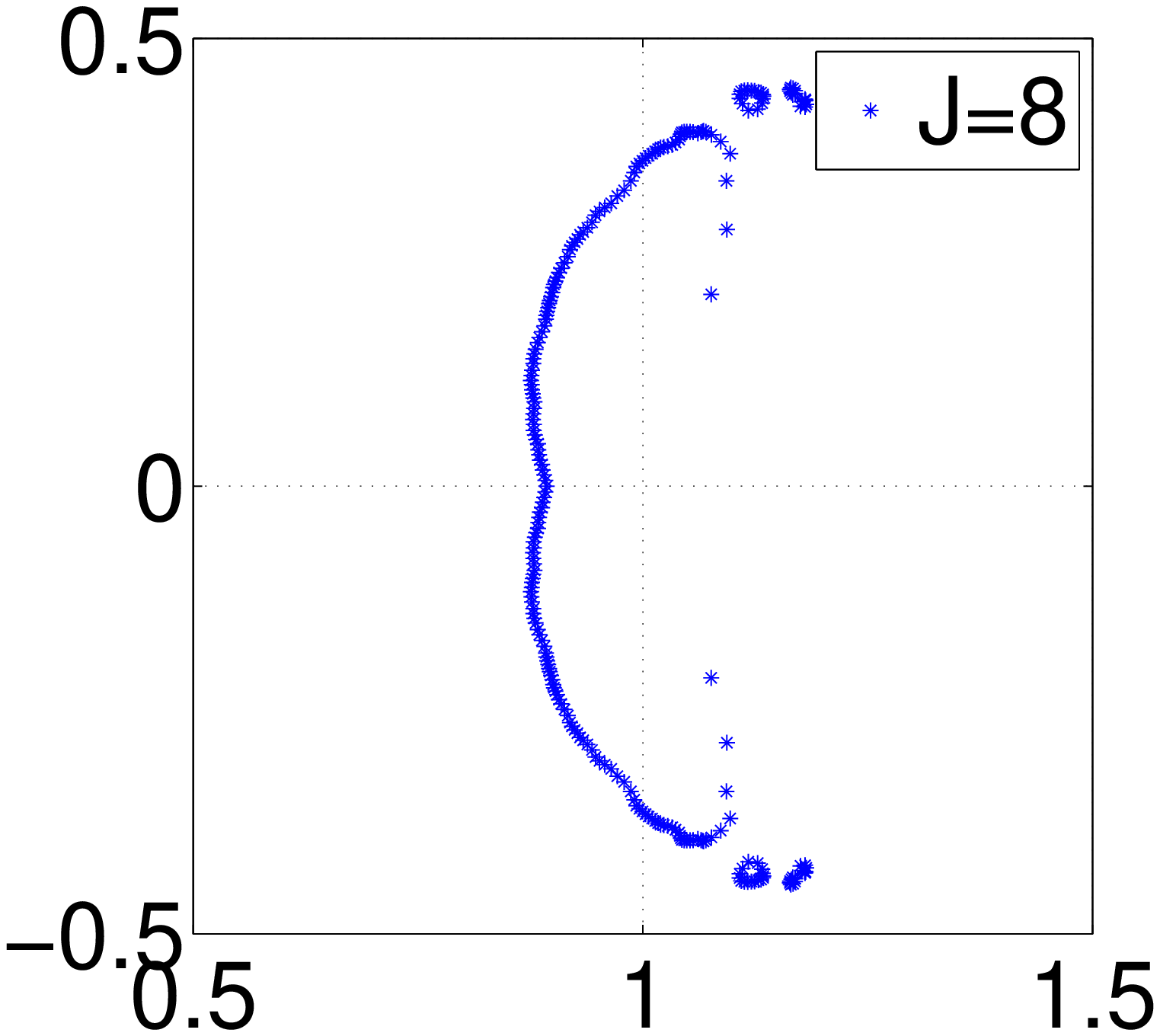}
\includegraphics[width=1.2in,height=1.2in,angle=0]{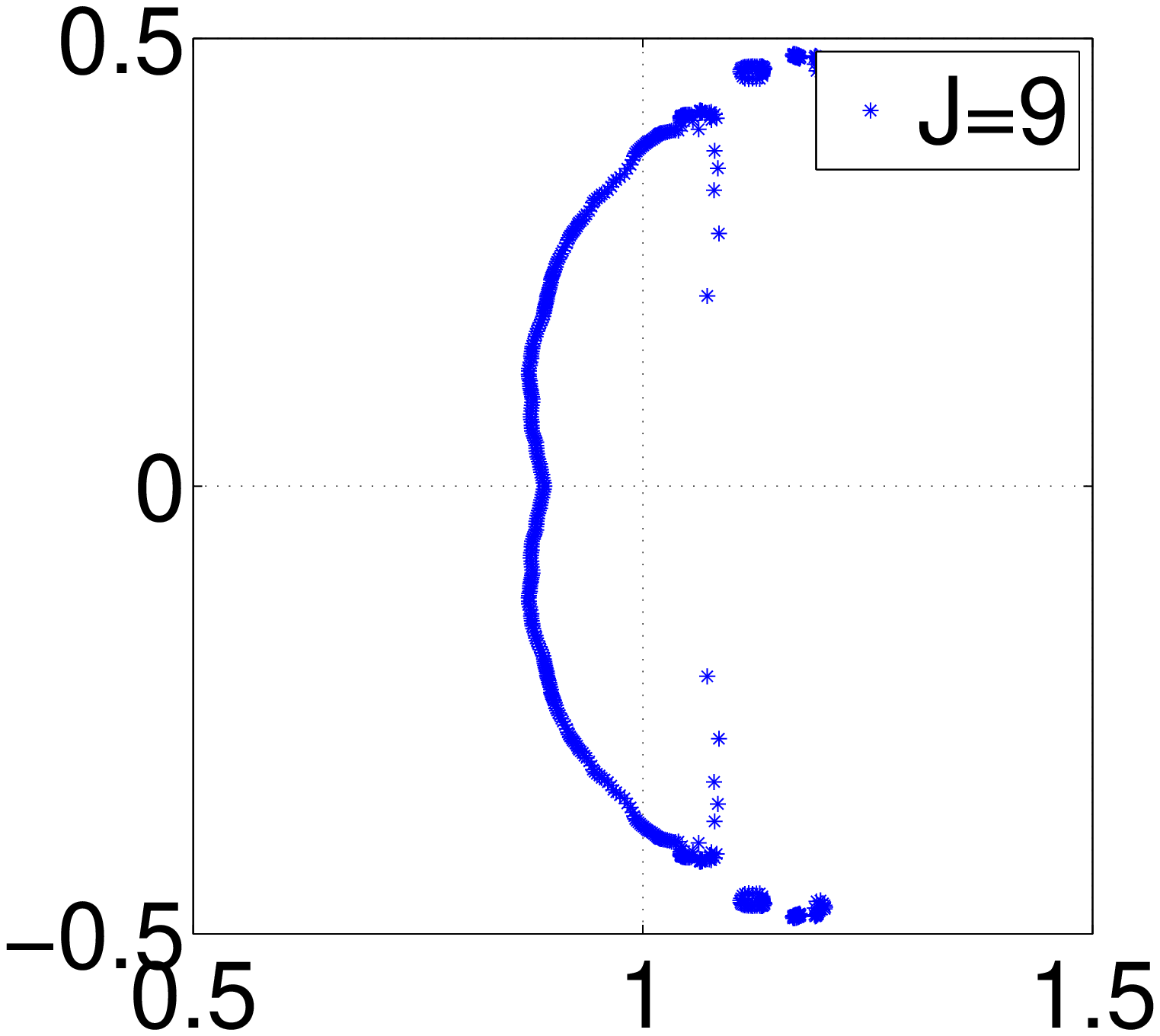}
\caption{Spectral distribution of the  matrix ${\bf A}$  (first two columns) and ${\bf D}{\bf W}^T{\bf A} {\bf W} {\bf D}$ (last two columns) for Example \ref{example1} with $q=0,\,\alpha=1.7$ and $\lambda=3$. }\label{fig1}
\end{center}
\end{figure}

\begin{example}\label{example2}
Consider the  time Caputo and  space tempered fractional  equation:
\begin{eqnarray}\label{solutioneq1}
{}_0^C{D}_t^{\gamma}u(x,t)=K_\gamma\,{}_a D_x^{\alpha,\lambda}u(x,t)+f(x,t)\quad 0<\gamma< 1,\,1<\alpha<2
\end{eqnarray}
with $u(0,t)=u(1,t)=0, K_\gamma=1$, where $g(x)$ and $f(x,t)$ are derived from the exact solution
\begin{eqnarray}
u(x,t)=\left(1+t^\beta\right)e^{-\lambda x}\left(x^3-x^2\right).
\end{eqnarray}
\end{example}

Let $S_h$ be the linear element space, and  define
\begin{eqnarray}\left({\bf L}_h\varphi, \chi\right)=-\left({}_{0}\mathbb{D}_x^{\frac{\alpha}{2},\lambda}\varphi,\,{}_{x}\mathbb{D}_1^{\frac{\alpha}{2},\lambda}\chi\right)
+\alpha\lambda^{\alpha-1}\left({}_{0}\mathbb{D}_x^{\frac{\alpha}{2},0}{\varphi},\,{}_{x}\mathbb{D}_1^{1-\frac{\alpha}{2},0}{\chi}\right)
+\lambda^{\alpha}\left(\varphi,\,\chi\right) \quad \forall \varphi,\chi\in S_h.
 \end{eqnarray}
 We use (\ref{eqnarra3}) to exactly do the integration in (\ref{eqnarraeq}), and approximate the corresponding ${\rm \Pi}^{\gamma,\beta}v$ with the CF and PC schemes, respectively. The numerical results for different $\gamma$ and $\alpha$ at $T=2$ are presented in Figure \ref{fig2}, where the straight lines (with the slope $-2$) give  a strong indication that the induced errors from the CF or PC approximation are negligible compared to the errors resulted from the finite-element discretization. Though one can only  assert that ${z^{\gamma-\beta}}\left({z^{\gamma}+K_{\gamma}{\bf L}_h}\right)^{-1}$ is analytic in $\mathbb{C}/(-\infty, 0]$ for $0<\gamma\le\frac{1}{2}$ while ${\bf L}_h$ is  positive define,  but  the  numerical experiments surprisingly show it can   done for all $\gamma$ without any problem.
\begin{figure}
\begin{center}
\includegraphics[width=2.2in,height=1.4in,angle=0]{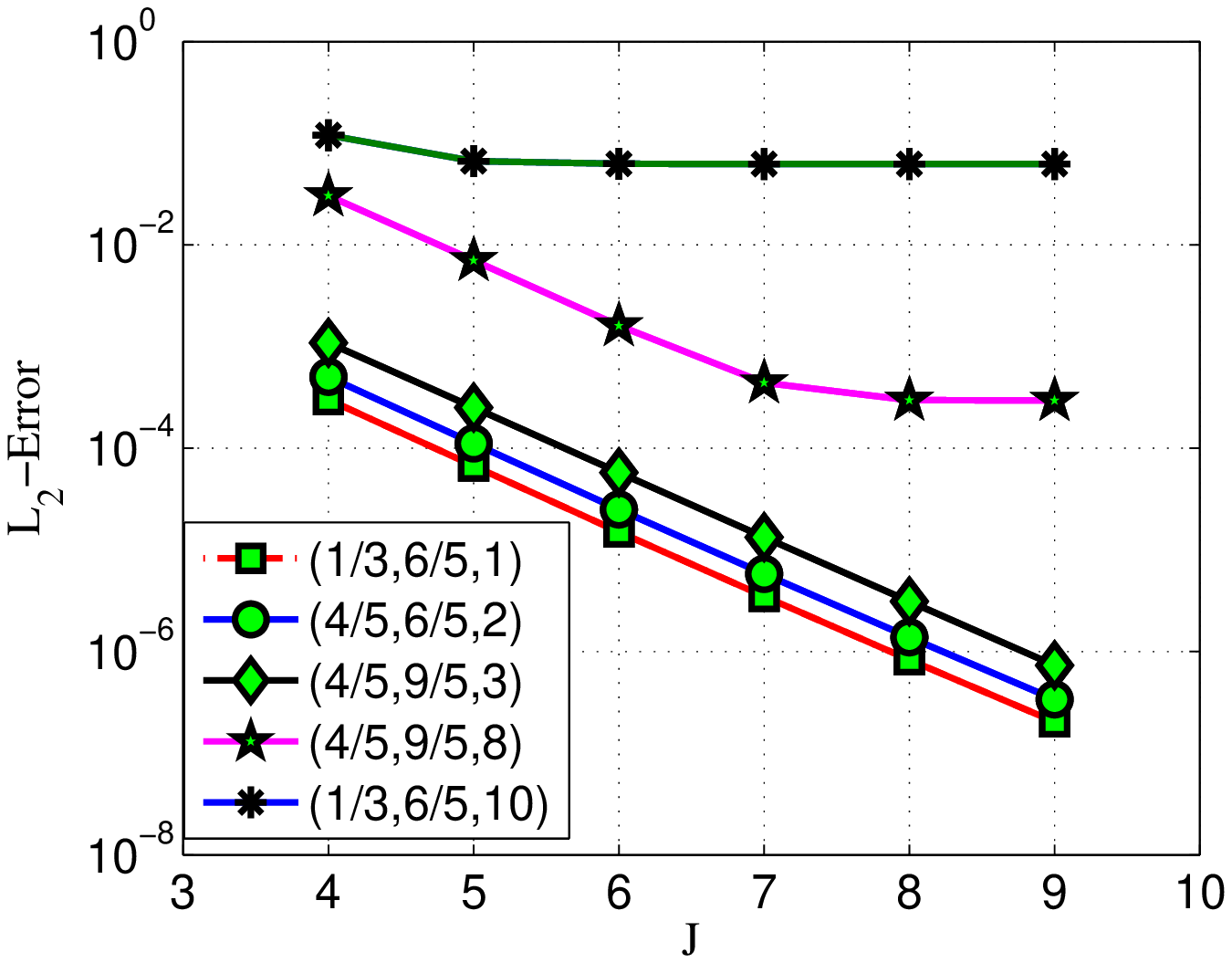}
\includegraphics[width=2.2in,height=1.4in,angle=0]{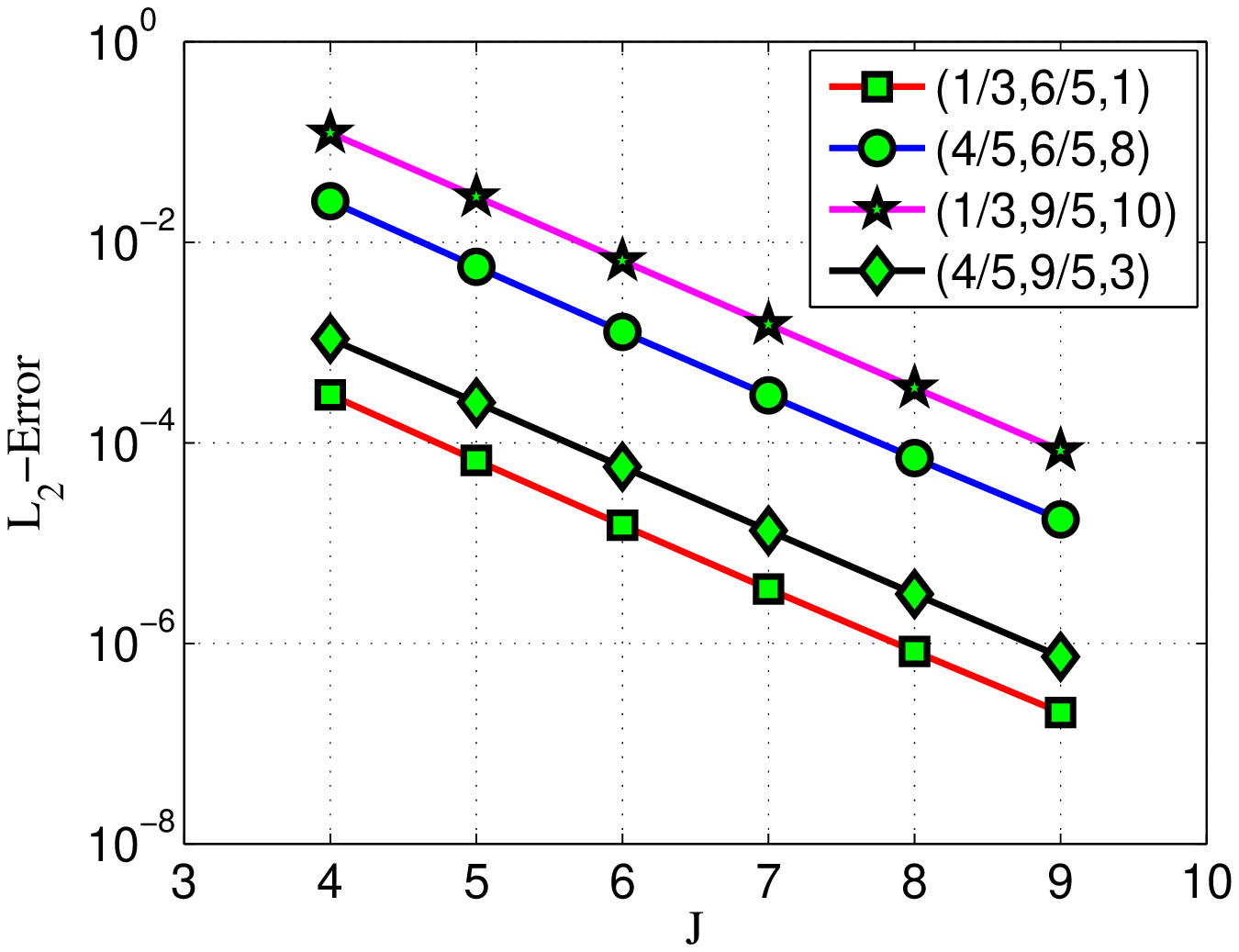}
\caption{Numerical results (semilogy coordinate) of the CF (left) and PC (right) schemes for Example \ref{example2} with $\lambda=3$ and $T=2$, where $(1/3,6/5,1)$ denotes $\gamma=1/3,\,\alpha=6/5$, and $\beta=1$. }\label{fig2}
\end{center}
\end{figure}

Since the existence of the non-real eigenvalues of ${\bf L}_h$, the DTI method fails here. Whereas the great strength of it lies in its simplicity and wide applied range. For example, if  one replaces the space derivative with the one-dimensional version of the fractional laplace  $-\left(-\triangle\right)^{\frac{\alpha}{2}}$ \cite{Yang:11,Yang:10} and defines ${\bf L}_h=\left(-\triangle_h\right)^{\frac{\alpha}{2}}, \,\left(\triangle_h\varphi, \chi\right)=-\left( \varphi^{\prime},\chi\prime\right)\, \forall \varphi,\chi\in S_h$, then the solution of (\ref{eqeeter4}) can be easily approximated  with the DTI method by just letting
   \begin{eqnarray}
{\rm \Pi}^{\gamma,\beta}v=\frac{t^{\beta-1}}{2\pi i}\int_{\mathcal{C}}E_{\gamma,\beta}(-K_{\gamma}t^{\gamma}z^{\frac{\alpha}{2}})\left(z{\bf I}-\left(-\triangle_h\right)\right)^{-1}v \,dz.
\end{eqnarray}
For $ g(x)=5sin(\pi x)\left(\cos(2\pi x)-1\right)$ and $f(x,t)=0$, the numerical results are presented in Table \ref{tab:2}.

\begin{table}\fontsize{8.5pt}{12pt}\selectfont
\begin{center}
 \caption{ The numerical results, solved by the DTI scheme, for Example \ref{example2}  with the spatial derivative $-\left(-\triangle\right)^{\frac{\alpha}{2}}$ and $T=5$. }
\begin{tabular} {c|cc|cc|cc|cc}  \hline
   $J$    &\multicolumn{2}{c|}{$\gamma=0.3,\alpha=1.2$}  &\multicolumn{2}{c|}{$\gamma=0.3,\alpha=1.8$} &\multicolumn{2}{c|}{$\gamma=0.8,\alpha=1.2$}  &\multicolumn{2}{c}{$\gamma=0.8,\alpha=1.8$}\\ \cline{2-9}
        & L2-Err&   CPU(s)     & L2-Err&  CPU(s)                               & L2-Err&   CPU(s)     & L2-Err&  CPU(s)   \\
      \hline
         $7$    & 2.7697e-05  & 0.0941  & 1.6318e-05   & 0.0562         &4.3443e-06   &0.0663     &  2.3827e-06  & 0.0557   \\
         $8$    &  6.9239e-06  & 0.1850  & 4.0796e-06   &0.1708         &1.0861e-06   &0.1770    & 5.9569e-07  & 0.1711     \\
         $9$    &  1.7310e-06  &1.4496  &1.0199e-06   & 1.4148          &2.7151e-07   &1.4591    &  1.4892e-07  & 1.4654 \\
    \hline
   \end{tabular}\label{tab:2}
\end{center}
\end{table}

\begin{example}\label{example3}
Consider the tempered time  fractional equation:
\begin{eqnarray}\label{example2eq}
{}_0^C\mathbb{D}_t^{\gamma,\lambda}u(x,t)= K_{\gamma} \frac{\partial ^2}{\partial x^2}u(x,t)+f(x,t)\quad 0<\gamma\le 1
\end{eqnarray}
with $ u(0,t)=u(1,t)=0, K_{\gamma}=1/\pi^2, g(x)=\sin (\pi x)$, and
\begin{eqnarray}
f(x,t)=w(t)e^{-\lambda t}g(x),\quad w(t)=\frac{\Gamma(\beta+1)t^{\beta-\gamma}}{\Gamma(\beta-\gamma+1)}+t^\beta+1.
\end{eqnarray}
Then the exact solution is $ u(x,t)=e^{-\lambda t}\left(t^\beta+1\right)\sin(\pi x)$.
\end{example}

Take the quadratic element space as $S_h$, which has the space convergence order $3$.
For $e^{\lambda(s-t)}{\bf P}_h f$, we use the quadratic interpolation in time; and the numerical performances are displayed in Table \ref{tab:3}. For comparison, we also
show the results of the $2-\gamma$ order $L_1$-time stepping scheme \cite{Deng:08} in last  two columns, i.e.,
\begin{eqnarray*}
{}_0^C\mathbb{D}_t^{\gamma,\lambda}v\left|_{t=t_k}\right.\approx \frac{e^{-\lambda t_k}}{\Gamma(1-\gamma)}\sum_{j=0}^{k-1}\int_{t_j}^{t_{j+1}}\left(t_k-s\right)^{-\gamma}\frac{\partial P_j}{\partial s}ds \quad k=1,\cdots, M=\lceil 2^{\frac{3J}{2-\gamma}}\rceil,
\end{eqnarray*}
where $P_j=e^{\lambda t_{j+1}}v(t_{j+1})+\frac{s-t_{j+1}}{\tau}\left(e^{\lambda t_{j+1}}v(t_{j+1})-e^{\lambda t_j}v(t_j)\right)$ and $\tau=\frac{T}{M}$.

The numerical results show that the time-direction errors of the CF, PC and DTI schemes are determined by the interpolation errors ($3$-order), which can be predesigned. Moreover, they are easy to do the parallel computing.  Since the CF scheme uses the least amount of quadrature points, it is fastest; and the PC scheme follows. The matlab ``eig" function  is used to obtain the extreme eigenvalues in the DTI scheme and the time is not included here.  In addition, if  $f(x,t)$ (w.r.t. time) is sufficiently smooth, a faster speed might be realized by the spectral  PC scheme. For example, if $\beta=9$,   the $16$ order Chebshev spectral  interpolation PC scheme with the time  $0.3078s,   0.1389s,    0.2287s$  is much better than the quadratic interpolation CF scheme with the time $0.1215s,    0.3248s,    1.3567s$; and the coefficients  $c_{M,j}$ in (\ref{eqnareee}) are roughly obtained by the ``chebfun.interp1" and ``poly" functions in the Chebfun project \cite{Driscoll:14}.
\begin{table}\fontsize{8.5pt}{12pt}\selectfont
\begin{center}
 \caption{ The numerical performance of Example \ref{example3} with $ \gamma=0.6, \lambda=1, \beta=4$, and $T=1$. Here ``CF,~~$2^J$" denotes that the CF scheme and $M=2^J$  equidistant partitions w.r.t. time are used; similar for the other ones. }
\begin{tabular} {c|cc|cc|cc|cc}  \hline
   $J$    &\multicolumn{2}{c|}{CF, \,$2^J$ }  &\multicolumn{2}{c|}{PC, \,$2^J$ } &\multicolumn{2}{c|}{DTI,\, $2^J$ }  &\multicolumn{2}{c}{$L_1$,\, $2^{2J}$}\\ \cline{2-9}
        & L2-Err&   CPU(s)     & L2-Err&  CPU(s)                               & L2-Err&   CPU(s)     & L2-Err&  CPU(s)   \\
    \hline
         $7$    &  4.4249e-08   & 0.1212  &4.4249e-08   &1.0614        & 4.4249e-08    &5.3500    & 1.2448e-07  &218.10     \\
         $8$    &  5.5308e-09   & 0.3569  &5.5303e-09   &3.9575        & 5.5303e-09    &12.452    & 1.5658e-08  & 5454.8     \\
         $9$    &   6.9161e-10  &1.4283  &6.9118e-10   &15.495         & 6.9117e-10     &25.808    & ----  &  $> 12$ hours            \\
    \hline
   \end{tabular}\label{tab:3}
\end{center}
\end{table}

\section{Conclusion}
The tempered anomalous diffusion attracts the wide interests of scientists. It is more close to reality in the sense that the physical space is bounded and the life span of the particles is finite. This paper focuses on providing the variational framework and efficient numerical implementation for the tempered PDEs describing the tempered anomalous diffusion. We first presented the variational properties of the tempered fractional derivatives,  which are used to establish the  Galerkin and Petrov-Galerkin method for solving the space tempered fractional differential equations. Meanwhile, we also studied the properties of the tempered fractional integrals,  which allow us to perform the theoretical analysis of the Perov-Galerkin method for the time tempered fractional equations. The efficient implementations, including the Galerkin and Petrov-Galerkin finite element method, the time  integrator, and the rational approximation method, are detailedly discussed. And the well performed numerical simulation results confirm the theoretical analysis and show the high efficiency of the schemes.

\def\ack{\section*{Acknowledgements}%
  \addtocontents{toc}{\protect\vspace{6pt}}%
  \addcontentsline{toc}{section}{Acknowledgements}}
\ack{The authors thank Xudong Wang for his help in the proof of Lemma \ref{lemmaeq2}. This work was supported by the National Natural Science Foundation of China under Grant  No. 11271173.}
\bibliographystyle{elsarticle-num}

\end{document}